\documentclass[12pt]{article} % use larger type; default would be 10pt

\usepackage{amsmath}
\usepackage{dcolumn}
\usepackage{endnotes}
\usepackage{tabularx}

\usepackage{graphicx}
\usepackage{epstopdf}
\usepackage{color}
\usepackage{amsthm}
\usepackage{amssymb}
 \usepackage{mathrsfs}
\usepackage{chemarrow}
 \usepackage{amssymb}

  \theoremstyle{plain}
\newtheorem{theorem}{Theorem}[section]

\newtheorem{lemma}[theorem]{Lemma}
\newtheorem{corollary}[theorem]{Corollary}

\theoremstyle{definition}
\newtheorem{definition}[theorem]{Definition}
\newtheorem{example}[theorem]{Example}

\usepackage[utf8]{inputenc} % set input encoding (not needed with XeLaTeX)

\usepackage{geometry} % to change the page dimensions
\usepackage{graphicx} % support the \includegraphics command and options

\usepackage{booktabs} % for much better looking tables
\usepackage{array} % for better arrays (eg matrices) in maths
\usepackage{paralist} % very flexible & customisable lists (eg. enumerate/itemize, etc.)
\usepackage{verbatim} % adds environment for commenting out blocks of text & for better verbatim
\usepackage{subfig} % make it possible to include more than one captioned figure/table in a single float
\usepackage{fancyhdr} % This should be set AFTER setting up the page geometry
\usepackage{sectsty}
\allsectionsfont{\sffamily\mdseries\upshape} % (See the fntguide.pdf for font help)
\usepackage[nottoc,notlof,notlot]{tocbibind} % Put the bibliography in the ToC
\usepackage[titles,subfigure]{tocloft} % Alter the style of the Table of Contents

\title{$(L,M)$-fuzzy  convex structures
\thanks{The project is supported by the National Natural Science Foundation of China (11371002) and Specialized
Research Fund for the Doctoral Program of Higher Education (20131101110048). }
}

\author{Fu-Gui Shi$^{1}$, Zhen-Yu Xiu$^{2}$\footnote{Corresponding author. Email:xyz198202@163.com}  \\
\small\it$^1$ School of Mathematics and Statistics, Beijing Institute of Technology,\\ \small\it Beijing 100081, China  \\
\small\it  $^2$College of Applied Mathematics, Chengdu University of Information Technology,\\ \small\it   Chengdu 610000, China\\
}

\date{} % Activate to display a given date or no date (if empty),
\begin{document}
\maketitle

\begin{abstract}
In this paper,
 the notion of $(L,M)$-fuzzy convex structures is
introduced.   It is a
generalization of  $L$-convex structures and $M$-fuzzifying convex structures.
In our  definition of $(L,M)$-fuzzy   convex structures,  each $L$-fuzzy
subset can be regarded as an $L$-convex set to some degree.    The notion of convexity preserving functions is also generalized to  lattice-valued case.   Moreover,  under the framework   of  $(L,M)$-fuzzy convex structures,  the concepts of   quotient structures, substructures and   products  are  presented  and  their fundamental  properties are discussed.   Finally, we  create a functor $\omega$ from $\mathbf{MYCS}$
 to $\mathbf{LMCS}$ and show that
% $\omega$ has a right-adjoint, hence
 there exists an adjunction
between $\mathbf{MYCS}$ and  $\mathbf{LMCS}$, where $\mathbf{MYCS}$ and   $\mathbf{LMCS}$  denote the category of $M$-fuzzifying convex structures,
%with their $I$-fuzzifying convexity preserving functions
 and the category of $(L,M)$-fuzzy convex structures,
%with their $(I,I)$-fuzzy convexity preserving functions
 respectively.
\end{abstract}

{\bf Keywords:}
%$(L,M)$-fuzzy convexity,
$(L,M)$-fuzzy convex structure,
$(L,M)$-fuzzy convexity preserving function,  quotient structures, substructures,
 products

%\begin{definition}
%\end{definition}
%
%\begin{lemma}
%\end{lemma}
%
%\begin{proof}
%\end{proof}
%\begin{theorem}
%\end{theorem}

\section{\bf Introduction  and preliminaries}

Convexity theory has been accepted to be of increasing importance in recent years
in the study of extremum problems in many areas of applied mathematics. The
concept of convexity which was mainly defined and studied in $\mathbb{R}^{n}$ in the pioneering works of Newton, Minkowski and others as described in \cite{Berger}, now finds a place in several other mathematical structures such as vector spaces, posets, lattices, metric spaces, graphs and median algebras.
 This development is motivated by not only the need for an abstract theory of convexity generalizing the classical theorems in $\mathbb{R}^{n}$ due to Helly, Caratheodory etc; but also by the necessity to unify geometric aspects of all these mathematical structures.   Abstract convexity theory is a branch of mathematics dealing with set-theoretic structures  satisfying axioms similar to that usual convex sets fulfill. Here, by ``usual convex sets",  we mean convex sets in real linear spaces. In a general setting, the axioms of abstract convexity
are the following:

(1) The empty set and the universe set are convex;

(2) The intersection of a nonempty collection of convex sets is convex;

(3) The union of a chain of convex sets is convex.

Clearly, usual convex sets have properties (1)-(3), but there are many other collections of sets, coming from various types of mathematical objects, that satisfy conditions (1)-(3), such as  convexities in lattices and in Boolean algebras \cite{Van,Varlet}, convexities in metric spaces and graphs  \cite{Lassak,Soltan}. Also, convex structures appeared naturally in topology, especially in the theory of supercompact spaces \cite{Mill}.

% Some more details about convex structures can be found in \cite{Van}.

The notion of a fuzzy subset was introduced by Zadeh \cite{Zadeh} and then fuzzy subsets have been applied to various branches of mathematics.
In 1994,   Rosa generalized the notion of a convex structure to a fuzzy  convex structure $(X, \mathfrak{C} )$  in \cite{{Rosa1},{Rosa2}}
and  $\mathfrak{C}$  was defined as a crisp family of  fuzzy subsets of a set $X$ satisfying certain axioms.  For convenience,  we  call this  fuzzy  convex structure an $I$-convex structure.
In 2009,   Maruyama generalized $I$-convex structures
 to  $L$-convex structures    in   \cite{Maruyama}, where  $L$ is  a completely distributive lattice.     % and  $\mathcal{C}$  was defined as a crisp family of  $L$-fuzzy subsets of a set $X$ satisfying certain axioms.
  In 2014,  Shi and Xiu \cite{Shi6} introduced  a new approach to the fuzzification of convex structures,  which  is called an $M$-fuzzifying convex structure.
 An $M$-fuzzifying  convex structure  is   a  pair $(X,\mathscr{C})$,  where $\mathscr{C}$   is  a mapping from $\mathbf{2}^{X}$ to $M$ satisfying three   axioms.  Recently,   Shi and Li \cite{shili} generalized the notion of restricted hull operators in classical convex spaces to $M$-fuzzifying restricted hull operators and used it to characterize $M$-fuzzifying convex structures. Pang and Shi \cite{PS} introduced several types of $L$-convex spaces, including stratified $L$-convex spaces, convex-generated $L$-convex spaces, weakly induced $L$-convex spaces and induced $L$-convex spaces and  discussed their relations from a categorical aspect.

In this paper, based on the  idea of \cite{Kubiak} and \cite{Sostak},  combining $L$-convex structures and $M$-fuzzifying convex structures  and   based on  complete  distributive  lattices $L$ and $M$,
 we present
 a more general approach to the fuzzification of
convex structures. More specifically,
we
define
an $(L,M)$-fuzzy convexity on a nonempty set $X$ by
means of a mapping $\mathcal{C} : L^X\rightarrow M$ satisfying three
axioms.   It is a
generalization of  $L$-convex structures and $M$-fuzzifying convex structures. Each $L$-fuzzy subset of $X$ can be regarded as an
$L$-convex set to some degree.
 %$(L,M)$-fuzzy convex structures are
%generalizations  of  $L$-convex structures and $M$-fuzzifying convex structures.

%When $\mathcal{C} $ is an $(L,M)$-fuzzy convexity  on $X$,
%$(X,\mathcal{C})$ is called an $(L,M)$-fuzzy convex structure. An
%$(L,\mathbf{2})$-fuzzy convex structure is also called an
%$L$-convex structure  in   \cite{Maruyama}.  An
%$(I,\mathbf{2})$-fuzzy convex structure  can be viewed as an
% $I$-convex structure in \cite{{Rosa1},{Rosa2}}.  A $(\mathbf{2},M)$-fuzzy convex structure is also
%called an $M$-fuzzifying convex structure. A crisp convex structure in \cite{Van}
%can be regarded as a $(\mathbf{2},\mathbf{2})$-fuzzy convex structure.

%without special instructions

Throughout this paper,  unless otherwise stated,    both $L$ and $M$  denote complete  distributive  lattices, $I=[0,1]$, $\mathbf{2}=\{0,1\}$  and $X$ is a nonempty
set. $L^X$ is the set of all $L$-fuzzy sets (or $L$-sets for short)
on $X$. We often do not distinguish a crisp subset $A$ of $X$ and
its  characteristic function $\chi_A$. The smallest element and the largest element in $L^X$ are denoted by $\chi_\emptyset$ and $\chi_X$, respectively.  The smallest element and the
largest element in $M (L)$ are denoted by $\bot_M (\bot_L)$ and $\top_M (\top_L)$,
respectively.  %The smallest element and the
%largest element in $L$ are denoted by $$ and $$,
%respectively.
We also adopt the convention that $\bigwedge\emptyset=\top_M$.

The binary relation $\prec$ in $M$ is defined as follows: for
$a,b\in M$, $a\prec b$ if and only if for every subset $D\subseteq
M$, the relation $b\leq  \sup D$ always implies the existence of
$d\in D$ with $a\leq  d$ \cite{Dwinger}. $\{a\in M: a\prec b\}$
is called the greatest minimal family of $b$ in the sense of
\cite{Wang}, denoted by $\beta(b)$.
 Moreover,  the binary relation $\prec^{op}$ in $M$ is defined as follows: for
$a,b\in M$, $a \prec^{op} b$ if and only if for every subset $D\subseteq
M$, the relation $\wedge D\leq a$ always implies the existence of
$d\in D$ with $d\leq b$.   $\{b\in M: a\prec^{op} b\}$
is called the greatest maximal family of $a$ in the sense of
\cite{Wang}, denoted by $\alpha(a)$.
%Moreover, for $b\in M$, we
%define $\alpha(b)=\{a\in M: a\prec^{op} b\}$.
 In a completely
distributive lattice $M$, there exist $\alpha(b)$ and $\beta(b)$ for
each $b\in M$, and $b=\bigvee\beta(b)=\bigwedge\alpha(b)$ (see
\cite{Wang}).

%In \cite{Wang}, Wang thought that $\beta(0)=\{0\}$ and
%$\alpha(1)=\{1\}$. In fact, it should be that $\beta(0)=\emptyset$
%and $\alpha(1)=\emptyset$.

For $a\in L$ and $A\in L^X$,  we use the
following notations:\\
(1) $A_{[a]}=\{x\in X: a\leq A(x)\}$.
(2) $A^{[a]}=\{x\in X: a\not\in \alpha(A(x))\}$.\\
(3) $A_{(a)}=\{x\in X: a\in \beta( A(x))\}$.

%When $L=[0,1]$, $A^{[a]}=A_{[a]}$ and  $A_{(a)}=A^{(a)}$.

%For $A\in L^X$ and $a\in L$, we use the
%following notations: $$A_{[a]}=\{x\in X: A(x)\geqslant a\},\ A_{(a)}=\{x\in X: a\in \beta( A(x))\}  $$
%$$
%A^{[a]}=\{x\in X: a\not\in \alpha(A(x))\},\ A^{(a)}=\{x\in X: A(x)\nleq a\}.
%$$

Some properties of these cut sets can be found in
\cite{Huang,Negoita,Shi1,Shi4}.

Let $f:X\rightarrow Y$ be a mapping. Define
$f^{\rightarrow}_{L}:L^{X}\rightarrow L^{Y}$ and
$f^{\leftarrow}_{L}:L^{Y}\rightarrow L^{X}$ by
$f^{\rightarrow}_{L}(A)(y)=\bigvee\limits_{f(x)=y}A(x)$ for $A\in L^{X}$
and $y\in Y$, and $f^{\leftarrow}_{L}(B)=B\circ f$ for $B\in L^{Y},$
respectively.

\begin{theorem} [\cite{Wang}]  For $\{a_i:i\in \Omega\}\subseteq M$,
\begin{enumerate}
\item[$(1)$] $\alpha\left(\bigwedge\limits_{i\in
\Omega}a_i\right)=\bigcup\limits_{i\in \Omega}\alpha(a_i)$, i.e.,
$\alpha$ is a $\bigwedge-\bigcup$ mapping.

\item[$(2)$] $\beta\left(\bigvee\limits_{i\in
\Omega}a_i\right)=\bigcup\limits_{i\in \Omega}\beta(a_i)$,  i.e.,
$\beta$ is a $\bigvee-\bigcup$ mapping.
\end{enumerate}
\end{theorem}

%For $a\in L$ and $D\subseteq X$, we define an $L$-fuzzy set
%$a\wedge D$  as follows:
%$$(a \wedge D)(x)=\left\{ \begin{array}{ll}
%a,& x\in D;\\
%\bot_L,& x\not\in D. \end{array}\right.$$

%Then for each $L$-fuzzy set
%$A$ in $L^X$, it follows that
% $$ A=\bigvee\limits_{a\in L}(a\wedge
%A_{[a]}).$$

\begin{theorem}[\cite{Huang,Shi4}]\label{11a} \label{12q} For each
$L$-fuzzy set $A$ in $L^X$, we have:\\
$(1)$  $ A=\bigvee\limits_{a\in L}(a\wedge
A_{[a]})=\bigwedge\limits_{a\in L}(a\vee A^{[a]})$.  $(2)$  $ \forall a\in L$, $A_{[a]}=\bigcap\limits_{b\in
\beta(a)}A_{[b]}=\bigcap\limits_{b\in
\beta(a)}A_{(b)}$. \\
$(3)$  $\forall a\in L$,  $A^{[a]}=\bigcap\limits_{a\in
\alpha(b)}A^{[b]}$.\ \ \ \ \ \ \  \  \  \  \  \   $(4)$  $ \forall a\in L$, $A_{(a)}=\bigcup\limits_{a\in
\beta(b)}A_{[b]}.$

\end{theorem}
\begin{theorem}[\cite{Shi1}]\label{12w}
For a family of $L$-fuzzy sets $\{A_i:i\in \Omega\}$ in $L^X$ and $a\in L$, we have: \\ %\begin{enumerate}
%\item[$(1)$] $ \left(\bigwedge\limits_{i\in \Omega}
%A_i\right)_{[a]}=\bigcap\limits_{i\in \Omega}\left(A_i\right)_{[a]}$.
%\item[$(2)$] $\left(\bigvee\limits_{i\in \Omega}
%A_i\right)_{(a)}=\bigcup\limits_{i\in \Omega}\left(A_i\right)_{(a)}$.
%\item[$(3)$] $ \left(\bigwedge\limits_{i\in
%\Omega} A_i\right)^{[a]}=\bigcap\limits_{i\in
%\Omega}\left(A_i\right)^{[a]}$.
% \end{enumerate}
$(1)$ $\left(\bigwedge\limits_{i\in \Omega}
A_i\right)_{[a]}=\bigcap\limits_{i\in \Omega}\left(A_i\right)_{[a]}$.
$(2)$ $\left(\bigvee\limits_{i\in \Omega}
A_i\right)_{(a)}=\bigcup\limits_{i\in \Omega}\left(A_i\right)_{(a)}$.\\
$(3)$ $ \left(\bigwedge\limits_{i\in
\Omega} A_i\right)^{[a]}=\bigcap\limits_{i\in
\Omega}\left(A_i\right)^{[a]}$.
\end{theorem}

\begin{definition}[\cite{Van}]\label{qwerd}
A subset $\mathbb{C}$ of $\mathbf{2}^X$   is  called
a convexity   if it satisfies the following conditions:
\begin{description}
 \item[(C1)] $\emptyset, X \in \mathbb{C}$;
\item[(C2)]   if $\{A_{i}:i\in \Omega  \} \subseteq \mathbb{C}$   is nonempty,
then $\bigcap \limits_{i\in \Omega}A_{i}\in \mathbb{C}$;
\item[(C3)]  if  $\{A_{i}:i\in \Omega  \} \subseteq \mathbb{C}$ is nonempty and totally ordered by inclusion,  then
 $\bigcup \limits_{i\in \Omega}A_{i}\in \mathbb{C}$.
\end{description}
The pair $(X,\mathbb{C})$ is calld a  convex structure and
the elements in $\mathbb{C}$ are called  convex sets.
%  Moreover,  if   $\mathcal{C} \subseteq  \mathbf{2}^{X}$ satisfies  (C1), (C2) and the following condition:
%\begin{description}
%\item[(C)]   $\mathcal{C}$ is stable for up-directed union, that is, if  $\{A_{i}: i\in \Omega \} \subseteq  \mathbf{2}^{X}$  is an up-directed set, then $\bigcup\limits_{i\in \Omega}A_{i}  \in \mathcal{C}$,
%\end{description}
%then  $(X,\mathcal{C})$ is also a convex structure.
\end{definition}

\begin{definition}[\cite{Maruyama}]  For a nonempty set $X$ and a subset $\mathfrak{C}$  of $L^X$,  $\mathfrak{C}$   is called
an $L$-convexity  if it satisfies the following conditions:
\begin{description}
 \item[(LC1)] $\chi _{\emptyset},\chi _{X} \in \mathfrak{C}$;

\item[(LC2)]   if $\{A_{i}:i\in \Omega  \} \subseteq \mathfrak{C}$   is nonempty,
then $\bigwedge \limits_{i\in \Omega}A_{i}\in \mathfrak{C}$;

\item[(LC3)]  if  $\{A_{i}:i\in \Omega  \} \subseteq \mathfrak{C}$ is nonempty and totally ordered by inclusion,  then
 $\bigvee \limits_{i\in \Omega}A_{i}\in \mathfrak{C}$.
\end{description}
If $\mathfrak{C}$  is  an  $L$-convexity on $X$, then the pair $(X,\mathfrak{C})$ is called an $L$-convex structure. When $L=\mathbf{2}$,  an  $L$-convexity is exactly an $I$-convex structure in \cite{Rosa1,Rosa2}.
\end{definition}

%\begin{definition}[\cite{Maruyama}]  For a nonempty set $X$ and a subset $\mathfrak{C}$
%
%of $L^X$,  $\mathfrak{C}$   is called
%an $L$-convexity  if it satisfies the following conditions:
%\begin{description}
% \item[(MC1)] $\chi _{\emptyset},\chi _{X} \in \mathfrak{C}$;
%
%\item[(MC2)]   if $\{A_{i}:i\in \Omega  \} \subseteq \mathfrak{C}$   is nonempty,
%then $\bigwedge \limits_{i\in \Omega}A_{i}\in \mathfrak{C}$;
%
%\item[(MC3)]  if  $\{A_{i}:i\in \Omega  \} \subseteq \mathfrak{C}$ is nonempty and totally ordered by inclusion,  then
% $\bigvee \limits_{i\in \Omega}A_{i}\in \mathfrak{C}$.
%\end{description}
%If $\mathfrak{C}$  is  an  $L$-convexity on $X$, then the pair $(X,\mathfrak{C})$ is called an $L$-convex structure. When $L=\mathbf{2}$,  an  $L$-convexity is exactly a fuzzy  convex structure in \cite{Rosa1,Rosa2}.
%\end{definition}

\begin{definition}[\cite{Shi6}]
 A  mapping $\mathscr{C}: \mathbf{2}^{X} \rightarrow M$ is called an $M$-fuzzifying  convexity
 on $X$
if it satisfies the following conditions:
\begin{description}
\item[(MYC1)] $\mathscr{C}(\emptyset)=\mathscr{C}(X)=\top_{M}$;
\item[(MYC2)] if $\{A_{i}:i\in \Omega \} \subseteq \mathbf{2}^{X}$   is nonempty, then
$\mathscr{C}( \bigcap\limits_{i\in \Omega}A_{i})\geq \bigwedge\limits_{i\in \Omega}\mathscr{C}(A_{i})$;
% \bigvee\limits_{B\subset C}[f(B)(|B|)]^{'} \ngeq  [f(C)(|C|)]^{'};\\
\item[(MYC3)]  if  $\{A_{i}:i\in \Omega \} \subseteq \mathbf{2}^{X}$    is
nonempty and totally ordered by inclusion, then $\mathscr{C}( \bigcup\limits_{i\in \Omega}A_{i})\geq \bigwedge\limits_{i\in \Omega}\mathscr{C}(A_{i})$.
\end{description}
If $\mathscr{C}$  is  an  $M$-fuzzifying  convexity on $X$, then the pair $(X,\mathscr{C})$ is called an $M$-fuzzifying convex structure.
\end{definition}

\begin{theorem}[\cite{Shi6}]  \label{cdax2}   A mapping $\mathscr{C}: {\bf2}^X \rightarrow M$ is an $M$-fuzzifying
convexity if and only if for each $a\in
M\backslash\{\bot_M\}$, $\mathscr{C}_{[a]}$ is a convexity.
 \end{theorem}

\begin{definition}[\cite{Shi6}]
 Let $\varphi: \mathbf{2}^{X}\rightarrow M$ be a mapping.  The $M$-fuzzifying  convex structure  $(X,\mathscr{C})$ generated by $\varphi$ is given by
  $$\forall A\in  \mathbf{2}^{X},\    \mathscr{C}_{\varphi}(A)=\bigwedge \{\mathscr{D}(A): \varphi\leq  \mathscr{D}\in \mathfrak{\mathfrak{G}}  \},$$ where  $\mathfrak{\mathfrak{G}}$ denotes all the $M$-fuzzifying convexities on $X$.  Then  $\varphi$ is called a subbase of   the $M$-fuzzifying convexity  $\mathscr{C}$.  Alternatively, we say that  $\varphi$ generates the convexity $ \mathscr{C}_{\varphi}$.
\end{definition}

%\begin{definition}[\cite{Fan}]   Let $X$ be a set and  $e : X\times X\rightarrow L$  be a mapping.  The pair $(X, e)$ is called an $L$-ordered set if for all $x, y, z\in X$,
%\begin{description}
%\item[(E1)]  $e(x, x) =\top_{L}$;
%\item[(E2)]  $e(x, y)\wedge e(y, z)\leq  e(x, z)$;
%\item[(E3)]  $e(x, y)\wedge e(y, x) =\top_{L}$ implies $x = y$.
%\end{description}
% An $L$-fuzzy  set $A$  on $X$  is an
%$L$-fuzzy lower set if $A(x) \geq  A(y)\wedge e(x,y)$ for $x,y\in X$.
%  An $L$-fuzzy  set $A$  on $X$  is an
%$L$-fuzzy upper  set if $A(x) \geq  A(y)\wedge e(y,x)$ for $x,y\in X$.
%\end{definition}

\begin{definition}[\cite{Ajmal}] Let $L$  be a lattice and
  $A$  a fuzzy subset of $L$.  Then $A$ is called a fuzzy sublattice of $L$ if for all $x, y \in L$,\\
(i)   $A(x\wedge y)\geq A(x) \wedge A(y)$,\\
(ii)  $A(x \vee y)\geq A(x) \wedge A(y)$.

 A fuzzy sublattice $A$ is said to be fuzzy convex if for every interval $[a, b]\subseteq L$
and for all $x \in [a, b]$,   $A(x)\geq A(a)\wedge  A(b)$.
\end{definition}

\begin{definition}[\cite{Berger}]
 Let $G$ be a group. A fuzzy subset $\lambda$ of $G$ is said to be a fuzzy
subgroup if\\
(1) $\lambda(xy)\geq\lambda(x) \wedge \lambda(y)$,\\
(2) $\lambda(x^{-1}) \geq \lambda(x)$.

Let $G$ be an ordered group. A fuzzy subgroup $\lambda$ of $G$ is said to be a
fuzzy convex subgroup if  for every interval $[a, b]\subseteq G$
and for all $x \in [a, b]$, we have $\lambda(x) \geq \lambda(a)\wedge \lambda(b)$.
\end{definition}

\section{$(L,M)$-fuzzy convex structures}

%In this section,  we extend an  $L$-convex structure  $(X,\mathcal{C} )$ in   \cite{Maruyama},   where $\mathcal{C}$  was defined as a crisp family of  $L^{X}$  satisfying three axioms,  to an $M$-subset of $L^{X}$ as follows.
In this section,   combining the concepts of $L$-convex structures and $M$-fuzzifying convex structures,  we introduce  a general approach to the fuzzification of
convex structures as follows.

\begin{definition}\rm  \label{123qwas}
A mapping $\mathcal{C}: L^X \rightarrow M$
is called an $(L,M)$-fuzzy convexity  on $X$ if it satisfies the
following three conditions:
%\newcounter{myctr}
%\begin{list}{\textbf{(LMC\arabic{myctr})}}{\usecounter{myctr}}
%
%\item  $\mathcal{C} (\chi_{\emptyset})=\mathcal{C} (\chi_{X})=\top_M$;
%
%\item   if $\{A_{i}:i\in \Omega \} \subseteq L^X$   is nonempty, then
%$\mathcal C   \left( \bigwedge\limits_{i\in \Omega}A_{i} \right)  \geq \bigwedge\limits_{i\in \Omega}\mathcal C(A_{i})$;
%
%\item  if  $\{A_i: i\in \Omega\}\subseteq L^X$    is nonempty and totally ordered by inclusion, then  $\mathcal C \left(\bigvee\limits_{i\in \Omega}A_i\right)
%\ge \bigwedge\limits_{i\in \Omega} \mathcal C (A_i)$.
%\end{list}

\begin{description}
\item[(LMC1)] $\mathcal{C} (\chi_{\emptyset})=\mathcal{C} (\chi_{X})=\top_M$;

\item[(LMC2)]    if $\{A_{i}:i\in \Omega \} \subseteq L^X$   is nonempty, then
$\mathcal C   \left( \bigwedge\limits_{i\in \Omega}A_{i} \right)  \geq \bigwedge\limits_{i\in \Omega}\mathcal C(A_{i})$;

\item[(LMC3)]  if  $\{A_i: i\in \Omega\}\subseteq L^X$    is nonempty and totally ordered by inclusion, then  $\mathcal C \left(\bigvee\limits_{i\in \Omega}A_i\right)
\ge \bigwedge\limits_{i\in \Omega} \mathcal C (A_i)$.
\end{description}

%An $(L,M)$-fuzzy convex space $\mathcal C$ is said to be stratified
%if and only if it satisfies the following condition:
%
%{\bf(LMS1)*} $\forall a\in L$, $\mathcal{C} (a\wedge \chi_{X})=\top_M$.

If $\mathcal C$ is an  $(L,M)$-fuzzy convexity,  then  $(X,\mathcal C)$
is called an $(L,M)$-fuzzy convex structure.
%(convexity space,  aligned space)

%An $(L,{\bf2})$-fuzzy convexity is also called an
%$L$-convexity, and an $(L,{\bf2})$-fuzzy convex structure is
%also called an $L$-convex structure.  A $({\bf2},M)$-fuzzy convexity is also called an $M$-fuzzifying convexity, and a $({\bf2},M)$-fuzzy convex structure is also called an $M$-fuzzifying convex structure.
%Obviously a crisp  convex structure can be regarded as a
%$({\bf2},{\bf2})$-fuzzy convex structure.
%When $\mathcal{C} $ is an $(L,M)$-fuzzy convexity  on $X$,
%$(X,\mathcal{C})$ is called an $(L,M)$-fuzzy convex structure.
An
$(L,\mathbf{2})$-fuzzy convex structure is  an
$L$-convex structure.   An
$(I,\mathbf{2})$-fuzzy convex structure  can be viewed as an $I$- convex structure.
  A $(\mathbf{2},M)$-fuzzy convex structure is
 an $M$-fuzzifying convex structure.   A crisp convex structure in \cite{Van}
can be regarded as a $(\mathbf{2},\mathbf{2})$-fuzzy convex structure.

If  $\mathcal C$  is an $(L,M)$-fuzzy convexity, then $\mathcal C(A)$
can be regarded as the degree to which $A$ is an $L$-convex set.
\end{definition}

%\begin{remark} If  a subset $\mathcal C$ of $L^X$  is  regarded as a mapping
%$\mathcal C : L^X \rightarrow {\bf2}$, then $\mathcal C$ is an
%$L$-convexity if and only if it satisfies the following
%conditions:
%
%\begin{list}{\rm\textbf{(LC\arabic{myctr})}}{\usecounter{myctr}}
%\item  $\chi _{\emptyset},\chi _{X} \in \mathcal{C}$;
%
%\item  if $\{A_{i}:i\in \Omega  \} \subseteq \mathcal{C}$   is nonempty,
%then $\bigwedge \limits_{i\in \Omega}A_{i}\in \mathcal{C}$;
%
%\item   if  $\{A_{i}:i\in \Omega  \} \subseteq \mathcal{C}$ is nonempty and totally ordered by inclusion,
%    then $\bigvee \limits_{i\in \Omega}A_{i}\in \mathcal{C}$.
%\end{list}
%
%Moreover, when $L={\bf2}$, a mapping $\mathcal{C}: {\bf2}^X  \rightarrow
%M$ is an $M$-fuzzifying  convexity if and only if it satisfies
%the following conditions:
%
%\begin{list}{\rm\textbf{(MC\arabic{myctr})}}{\usecounter{myctr}}
%\item  $\mathcal{C}(\emptyset)=\mathcal{C}(X)=\top_M$;
%
%\item   if $\{A_{i}:i\in \Omega \} \subseteq 2^X$   is nonempty, then
%$\mathcal C   \left( \bigcap\limits_{i\in \Omega}A_{i} \right)  \geq \bigwedge\limits_{i\in \Omega}\mathcal C(A_{i})$;
%
%\item  if  $\{A_i: i\in \Omega\}\subseteq 2^X$    is nonempty and totally ordered by inclusion, then  $\mathcal C \left(\bigcup\limits_{i\in \Omega}A_i\right)
%\ge \bigwedge\limits_{i\in \Omega} \mathcal C (A_i)$.
%\end{list}
%\end{remark}

Next we give some examples  of  $(L,M)$-fuzzy convex structures, $L$-convex structures and $M$-fuzzifying  convex structures, respectively.

\begin{example}
Let a mapping $\mathcal{T}: L^{X}\rightarrow M$ be an $(L,M)$-fuzzy topology in  \cite{Kubiak, Sostak}.  If it satisfies the following conditions:
%\begin{description}
%\item[(FOA1)]  $\mathcal{T}(\emptyset) = \mathcal{T}(X) = 1$;
%\item[(FOA2)]  $A_{1}, A_{2}\in  \mathbf{2}^{X}$, $\mathcal{T}(A_{1}\cap A_{2})\geq \mathcal{T}(A_{1}) \wedge  \mathcal{T}(A_{2})$;
%\item[(FOA3)]  $\forall \{A_{j}\}_{j\in J}\subseteq \mathbf{2}^{X}$,  $\mathcal{T}(\bigcup_{j\in J}A_{j}) \geq \bigwedge _{j\in J}\mathcal{T}(A_{j})$.
%\end{description}
% And $(X,\mathcal{T})$ is called a fuzzifying topological space. Furthermore,
%if an $(L,M)$-fuzzy topology $\mathcal{T}$ satisfies:
\begin{description}
 \item[(S)]  $\forall \{A_{j}\}_{j\in J}\subseteq L^{X}$, $\mathcal{T}(\bigwedge_{j\in J}A_{j}) \geq \bigwedge _{j\in J}\mathcal{T}(A_{j})$, \end{description}
  then $\mathcal{T}$ is called a saturated $(L,M)$-fuzzy topology, and $(X,\mathcal{T})$ is called an Alexandroff $(L,M)$-fuzzy  topological space.

 We can see that an  Alexandroff  $(L,M)$-fuzzy topological space  $(X,\mathcal{T})$  is an $(L,M)$-fuzzy convex structure.

 When $L=\mathbf{2}$ and $M=I$,   an  Alexandroff  $(L,M)$-fuzzy topological space  $(X,\mathcal{T})$ is an   Alexandroff  fuzzifying topological space in \cite{Fang2,Ying1} and  it is an example of  $M$-fuzzifying  convex structures.
  \end{example}

%\begin{exam}[\cite{Fang2,Ying1}]\label{567}
%Let X be a universe of discourse. A mapping $\mathcal{T}: \mathbf{2}^{X}\rightarrow I$ is called a fuzzifying topology
%on $X$, if it satisfies the following conditions:
%\begin{description}
%\item[(FOA1)]  $\mathcal{T}(\emptyset) = \mathcal{T}(X) = 1$;
%\item[(FOA2)]  $A_{1}, A_{2}\in  \mathbf{2}^{X}$, $\mathcal{T}(A_{1}\cap A_{2})\geq \mathcal{T}(A_{1}) \wedge  \mathcal{T}(A_{2})$;
%\item[(FOA3)]  $\forall \{A_{j}\}_{j\in J}\subseteq \mathbf{2}^{X}$,  $\mathcal{T}(\bigcup_{j\in J}A_{j}) \geq \bigwedge _{j\in J}\mathcal{T}(A_{j})$.
%\end{description}
% And $(X,\mathcal{T})$ is called a fuzzifying topological space. Furthermore, if a fuzzifying topology $\mathcal{T}$ satisfies:
% \begin{description}
% \item[(SOA2)]  $\forall \{A_{j}\}_{j\in J}\subseteq \mathbf{2}^{X}$, $\mathcal{T}(\bigcap_{j\in J}A_{j}) \geq \bigwedge _{j\in J}\mathcal{T}(A_{j})$, \end{description}
%  then $\mathcal{T}$ is called a saturated fuzzifying topology, and $(X,\mathcal{T})$ is called an Alexandroff  fuzzifying topological space.
%
% We can see that an  Alexandroff  fuzzifying topological space  $(X,\mathcal{T})$  is an $I$-fuzzifying  convex structure.
%\end{exam}

\begin{example}[\cite{Fang1}]  \label{asdf} An $I$-fuzzified set of all upper sets of a fuzzy preordered set $(X,R)$ is a map $\nabla(R): I^{X} \rightarrow I$ defined by
$$\forall U\in I^{X},\  \nabla(R)(U)=\bigwedge\limits_{(x, y)\in X \times X}R(x, y)\rightarrow(U(x)\rightarrow U(x)).$$
For a given fuzzy preorder $R$ on $X$, $\nabla(R)$, the $I$-fuzzified set of all upper sets of $(X,R)$ has the following properties: for all $\mathcal{F}\subseteq I^{X}$, $U,V \in I^{X}$ and $\lambda\in [0,1]$, \\
(i)  $\nabla(R)(\lambda) = 1$  for every constant mapping $\lambda$ from $X$ to $[0,1]$;\\
(ii) $\bigwedge \nabla(R)(\mathcal{F})\leq \nabla(R)( \bigwedge\mathcal{F})$
 where  $\nabla(R)(\mathcal{F})=\{\nabla(R)(U)| U \in \mathcal{F}\}$;\\
(iii)
 $\bigwedge \nabla(R)(\mathcal{F})\leq \nabla(R)(\bigvee \mathcal{F})$. \\
 We can see that  $\nabla(R)$  satisfies (LMC1)-(LMC3) and then
$(X, \nabla(R))$ is an $(L,M)$-fuzzy convex structure,  where $M=L=I$.
\end{example}

\begin{example}
Define a mapping $\mathcal{C}: L^{\mathbb{R}^{n}}\rightarrow M$ by
 $$\forall A\in L^{\mathbb{R}^{n}},\
 \mathcal{C}(A)=\bigwedge\limits_{\lambda \in [0,1]} \bigwedge\limits_{(x, y)\in \mathbb{R}^{n} \times \mathbb{R}^{n}}(A(x)\wedge A(y))\rightarrow A(\lambda x + (1-\lambda)y),$$
where the binary function $\rightarrow$ is defined  as follows: for $a,b,c\in L$,
$a\rightarrow b=\bigvee\{c\in L: a\wedge c\leq b \}$.
 Then $(X, \mathcal{C})$ is an $(L,M)$-fuzzy convex structure.  Next we show that $\mathcal{C}$ satisfies (LMC1)-(LMC3).

%Moreover,
%We can see that it is an example for an $(I,I)$-fuzzy convex structure.
(LMC1) Clearly,  $\mathcal{C}(\chi_\emptyset)=\mathcal{C}(\chi_ {\mathbb{R}^{n}})=\top_{M}$.

(LMC2) For any nonempty set $\{A_{i}:i\in \Omega \} \subseteq L^{\mathbb{R}^{n}}$, we have
$$\begin{array}{lll}
&&\mathcal{C}(\bigwedge\limits_{i\in \Omega}A_{i})\\
&=&\bigwedge\limits_{\lambda \in [0,1]} \bigwedge\limits_{(x, y)\in \mathbb{R}^{n} \times \mathbb{R}^{n}}((\bigwedge\limits_{i\in \Omega}A_{i})(x)\wedge (\bigwedge\limits_{i\in \Omega}A_{i})(y))\rightarrow (\bigwedge\limits_{i\in \Omega}A_{i})(\lambda x + (1-\lambda)y)\\
&=&\bigwedge\limits_{\lambda \in [0,1]} \bigwedge\limits_{(x, y)\in \mathbb{R}^{n} \times \mathbb{R}^{n}}(\bigwedge\limits_{i\in \Omega}A_{i}(x)\wedge \bigwedge\limits_{i\in \Omega}A_{i}(y))\rightarrow \bigwedge\limits_{i\in \Omega}A_{i}(\lambda x + (1-\lambda)y)\\
&=&\bigwedge\limits_{\lambda \in [0,1]} \bigwedge\limits_{(x, y)\in \mathbb{R}^{n} \times \mathbb{R}^{n}}\bigwedge\limits_{i\in \Omega}((\bigwedge\limits_{i\in \Omega}A_{i}(x)\wedge \bigwedge\limits_{i\in \Omega}A_{i}(y))\rightarrow A_{i}(\lambda x + (1-\lambda)y))\\
&\geq &\bigwedge\limits_{\lambda \in [0,1]} \bigwedge\limits_{(x, y)\in \mathbb{R}^{n} \times \mathbb{R}^{n}}\bigwedge\limits_{i\in \Omega}(A_{i}(x)\wedge A_{i}(y))\rightarrow A_{i}(\lambda x + (1-\lambda)y)\\
&= &\bigwedge\limits_{i\in \Omega} \bigwedge\limits_{\lambda \in [0,1]} \bigwedge\limits_{(x, y)\in \mathbb{R}^{n} \times \mathbb{R}^{n}}((A_{i}(x)\wedge A_{i}(y))\rightarrow A_{i}(\lambda x + (1-\lambda)y))\\
&= &\bigwedge\limits_{i\in \Omega}\mathcal{C}(A_{i}).
\end{array}$$
The proof of (LMC3) is similar to that of (LMC2) and is omitted.
%For any set $\{A_{i}:i\in \Omega \} \subseteq I^{\mathbb{R}^{n}}$, which  is nonempty and totally ordered by inclusion, we have
%$$\begin{array}{lll}
%&&\mathcal{C}(\bigvee \limits_{i\in \Omega}A_{i})\\
%&=&\bigwedge\limits_{\lambda \in [0,1]} \bigwedge\limits_{(x, y)\in \mathbb{R}^{n} \times \mathbb{R}^{n}}((\bigvee\limits_{i\in \Omega}A_{i})(x)\wedge (\bigvee\limits_{i\in \Omega}A_{i})(y))\rightarrow (\bigvee\limits_{i\in \Omega}A_{i})(\lambda x + (1-\lambda)y)\\
%&=&\bigwedge\limits_{\lambda \in [0,1]} \bigwedge\limits_{(x, y)\in \mathbb{R}^{n} \times \mathbb{R}^{n}}(\bigvee\limits_{i\in \Omega}A_{i}(x)\wedge \bigvee\limits_{i\in \Omega}A_{i}(y))\rightarrow \bigvee\limits_{i\in \Omega}A_{i}(\lambda x + (1-\lambda)y)\\
%&=&\bigwedge\limits_{\lambda \in [0,1]} \bigwedge\limits_{(x, y)\in \mathbb{R}^{n} \times \mathbb{R}^{n}}\bigwedge\limits_{i\in \Omega}(A_{i}(x)\wedge A_{i}(y))\rightarrow  \bigvee\limits_{i\in \Omega}A_{i}(\lambda x + (1-\lambda)y)\\
%&\geq &\bigwedge\limits_{\lambda \in [0,1]} \bigwedge\limits_{(x, y)\in \mathbb{R}^{n} \times \mathbb{R}^{n}}\bigwedge\limits_{i\in \Omega}(A_{i}(x)\wedge A_{i}(y))\rightarrow A_{i}(\lambda x + (1-\lambda)y)\\
%&= &\bigwedge\limits_{i\in \Omega} \bigwedge\limits_{\lambda \in [0,1]} \bigwedge\limits_{(x, y)\in \mathbb{R}^{n} \times \mathbb{R}^{n}}(A_{i}(x)\wedge A_{i}(y))\rightarrow A_{i}(\lambda x + (1-\lambda)y)\\
%&= &\bigwedge\limits_{i\in \Omega}\mathcal{C}(A_{i}).
%\end{array}$$
\end{example}

%\begin{example} Let $X$ be a nonempty set and $\mathcal{C}: I^{X}\rightarrow I$ be a
%mapping  defined by
%$${\mathcal{C}}(A)=\left\{
%\begin{array}{cll}
%1,& A\in \{\chi_\emptyset,\chi_X\};\\
%0.5,& A\not\in \{\chi_\emptyset,\chi_X\}. \end{array}\right. $$ Then
%it is easy to prove that $(X, \mathcal{C})$ is an $(I,I)$-fuzzy
%convex structure.  If $A\in I^{X}$ with $A\not\in
%\{\chi_\emptyset,\chi_X\}$, then 0.5 is the degree to which $A$ is
% an $I$-convex set.\end{example}

When $M=\mathbf{2}$, we obtain the following example.

\begin{example}[\cite{Maruyama}]   An $L$-fuzzy set $\mu$ on $\mathbb{R}^{n}$
is an $L$-fuzzy convex set on $\mathbb{R}^{n}$
iff
$\mu(rx+(1-r)y)\geq\mu(x)\wedge\mu(y)$
for any $x,$ $y\in \mathbb{R}^{n}$
and for any $r\in[0,1]$. $\mathfrak{C}_{L}$
denotes the set of all $L$-fuzzy convex sets on $\mathbb{R}^{n}$. Then $(\mathbb{R}^{n},\mathfrak{C}_{L})$
is an  $L$-convex structure.
\end{example}
%\begin{example}Let  $(X, e)$ be an $L$-ordered set.  An $L$-fuzzy  set A  on $X$  is
%$L$-fuzzy order convex if  $A(z)\geq  A(x)\wedge A(y)\wedge e(x,z) \wedge e(z,y)$ for $x,y,z\in X$. Let  $\mathfrak{C}$
%denote the set of all $L$-fuzzy order convex sets on $X$.
% It is easy to verify  that $(X, \mathfrak{C})$ satisfies (LC1)-(LC3) and $(X, \mathfrak{C})$ is an  $L$-convex structure.
%\end{example}
%
%\begin{example}Let    $(X, e)$ be an $L$-ordered set.  Let  $\mathfrak{C}_{1}$ and  $\mathfrak{C}_{2}$
%denote the set of all $L$-fuzzy lower sets  and $L$-fuzzy upper sets on $X$, respectively.
% It is easy to verify  that $\mathfrak{C}_{1}$ and $\mathfrak{C}_{2}$
%are  $L$-convexities. We can call them an $L$-lower convexity and  an $L$-upper convexity, respectively.
%\end{example}

\begin{example}
Let  $\mathfrak{C}$ denote the set of all fuzzy  convex sublattices  on $L$.  It is easy to show that $\mathfrak{C}$ is an $I$-convexity and $(L, \mathfrak{C})$ is an  $I$-convex structure.
\end{example}
\begin{example}
Let G be an ordered group,  and  let  $\mathfrak{C}$ denote the set of all fuzzy convex subgroup  on $G$.  Then we can see that $\mathfrak{C}$ is an  $I$-convexity and $(G, \mathfrak{C})$ is an  $I$-convex structure.
\end{example}

%\begin{example} Let $(X,\mathcal{C})$ be a crisp convex structure.
%Define $\chi_{\mathcal{C}}: {\bf2}^X \rightarrow I$ by
%$$\chi_{\mathcal{C}}(A)=\left\{
%\begin{array}{cll}
%1,& A\in \mathcal{C};\\
%0,& A\not\in \mathcal{C}. \end{array}\right. $$ Then it is easy to prove
%that $(X, \chi_{\mathcal{C}})$ is an $I$-fuzzifying convex structure.
%\end{example}
%
%\begin{example} Let $X$ be a nonempty set and $\mathcal{C}: {\bf2}^{X}\rightarrow I$ be a
%mapping  defined by
%$${\mathcal{C} }(A)=\left\{
%\begin{array}{cll}
%1,& A\in \{\emptyset,X\};\\
%0.5,& A\not\in \{\emptyset,X\}. \end{array}\right. $$ Then it is
%easy to prove that $(X, {\mathcal{C} })$ is an $I$-fuzzifying
%convex structure. If $A\in {\bf2}^{X}$ with $A\not\in
%\{\emptyset,X\}$, then 0.5 is the degree to which $A$ is a
%convex set.\end{example}

%\begin{pro}  Let $(X,\mathcal{C})$ be an $(L,M)$-fuzzy measurable spaces.
%Then for any $\{A_n: n\in \mathbb{N}\}\subseteq L^X$,
%$\mathcal{C}\left(\bigwedge\limits_{n\in \mathbb{N}}A_n\right) \ge
%\bigwedge\limits_{n\in \mathbb{N}}\mathcal{C} (A_n)$.
% \end{pro}
%
%\proof{This can be proved from the fact:
%$$ \mathcal{C}\left(\bigwedge\limits_{n\in
%\mathbb{N}}A_n\right)=\mathcal{C}\left(\bigvee\limits_{n\in
%\mathbb{N}}(A_n)'\right)
% \ge \bigwedge\limits_{n\in \mathbb{N}}\mathcal{C}
%\left((A_n)'\right)=\bigwedge\limits_{n\in \mathbb{N}}\mathcal{C}
%(A_n).$$ }

The next two theorems  give characterizations of an $(L,M)$-fuzzy
convexity.

\begin{theorem}\label{23456}   A mapping $\mathcal{C} : L^X\rightarrow M$ is an $(L,M)$-fuzzy
convexity  if and only if for each $a\in
M\backslash\{\bot_M\}$, $\mathcal{C}_{[a]}$ is an $L$-convexity .
\end{theorem}

\begin{proof} The proof is obvious and is omitted.
\end{proof}

\begin{theorem}\label{mhn} %If $M$ is completely distributive, then  a
A mapping
$\mathcal{C}: L^X\rightarrow M$ is an $(L,M)$-fuzzy convexity if
and only if for each $a\in \alpha(\bot_M)$, $\mathcal{C}^{[a]}$ is an
$L$-convexity.
\end{theorem}
\begin{proof}
Sufficiency.
{\bf(LMC1)}  For each $a\in \alpha(\bot_M)$,  $\chi_\emptyset, \chi_X\in   \mathcal{C}^{[a]}$. We have   $\mathcal{C}(\chi_\emptyset)=\mathcal{C}(\chi_X)=\top_M$.

{\bf(LMC2)} Let  $\{A_{i}|i\in \Omega \} \subseteq L^{X}$  be nonempty,
and for $a\in \alpha(\bot_M)$,
$a \notin \alpha(\bigwedge \limits_{i\in \Omega}\mathcal{C}(A_{i}))$.
Thus $a \notin \bigcup\limits_{i\in \Omega}\alpha(\mathcal{C}(A_{i}))$. We know that $a \notin \alpha(\mathcal{C}(A_{i}))$ and then $A_{i}\in  \mathcal{C}^{[a]}$ for each $i\in \Omega$.  Since for each $a\in \alpha(\bot_M)$,
$\mathcal{C}^{[a]}$ is an $L$-convexity,  $\bigwedge\limits_{i\in \Omega}A_{i}\in \mathcal{C}^{[a]}$, that is, $a\notin  \alpha(\mathcal{C}(\bigwedge\limits_{i\in \Omega}A_{i}))$. Therefore $\mathcal{C}(\bigwedge\limits_{i\in \Omega}A_{i})\geq \bigwedge \limits_{i\in \Omega}\mathcal{C}(A_{i})$.

{\bf(LMC3)} Let $\{A_{i}|i\in \Omega \} \subseteq L^{X}$ be nonempty and totally ordered by inclusion, and let $a \notin \alpha(\bigwedge\limits_{i\in \Omega}\mathcal{C}(A_{i}))$ for $a\in \alpha(\bot_M)$.
Thus $a \notin \bigcup\limits_{i\in \Omega}\alpha(\mathcal{C}(A_{i}))$. We know that $a \notin \alpha(\mathcal{C}(A_{i}))$ and then $A_{i}\in  \mathcal{C}^{[a]}$ for each $i\in \Omega$.  Since for each $a\in \alpha(\bot_M)$,
$\mathcal{C}^{[a]}$ is an $L$-convexity,  $\bigvee\limits_{i\in \Omega}A_{i}\in \mathcal{C}^{[a]}$, that is, $a\notin  \alpha(\mathcal{C}(\bigvee\limits_{i\in \Omega}A_{i}))$. Therefore $\mathcal{C}(\bigvee\limits_{i\in \Omega}A_{i})\geq \bigwedge \limits_{i\in \Omega}\mathcal{C}(A_{i})$.

 Necessity. Suppose that $\mathcal{C} : L^X\rightarrow M$ is an
$(L,M)$-fuzzy convexity and $a\in \alpha(\bot_M)$. Now we
prove that $\mathcal{C}^{[a]}$ is an $L$-convexity.

{\bf(LC1)} By $\mathcal{C}(\chi_\emptyset)=\mathcal{C}(\chi_X)=\top_M$ and
$\alpha(\top_M)=\emptyset$, we know that $a\not\in
\alpha(\mathcal{C}(\chi_\emptyset))$ and $a\not\in
\alpha(\mathcal{C}(\chi_X))$. This implies $\chi_\emptyset, \chi_X \in
\mathcal{C}^{[a]}$.

{\bf(LC2)} If $\{A_i:i\in \Omega\}\subseteq \mathcal{C}^{[a]}$, then
$\forall i\in \Omega$, $a\not\in \alpha(\mathcal{C}(A_i))$. Hence
$a\not\in \bigcup\limits_{i\in \Omega}\alpha(\mathcal{C}(A_i))$. By
$\mathcal{C}(\bigwedge\limits_{i\in \Omega}A_i)\ge \bigwedge\limits_{i\in
\Omega}\mathcal{C}(A_i)$,  we know that
$\alpha\left(\mathcal{C}\left(\bigwedge\limits_{i\in \Omega}A_i\right)\right) \subseteq
\alpha\left(\bigwedge\limits_{i\in
\Omega}\mathcal{C}(A_i)\right)=\bigcup\limits_{i\in
\Omega}\alpha(\mathcal{C}(A_i))$.  This shows $a\not\in
\alpha\left(\mathcal{C}\left(\bigwedge\limits_{i\in
\Omega}A_i\right)\right)$. Therefore $\bigwedge\limits_{i\in
\Omega}A_i\in \mathcal{C}^{[a]}$.

{\bf(LC3)} If $\{A_i:i\in \Omega\}\subseteq \mathcal{C}^{[a]}$  is nonempty and totally ordered by inclusion,  then
$\forall i\in \Omega$, $a\not\in \alpha(\mathcal{C}(A_i))$. Hence
$a\not\in \bigcup\limits_{i\in \Omega}\alpha(\mathcal{C}(A_i))$. By
$\mathcal{C}(\bigvee\limits_{i\in \Omega}A_i)\ge \bigwedge\limits_{i\in
\Omega}\mathcal{C}(A_i)$, we know that
$\alpha\left(\mathcal{C}\left(\bigvee\limits_{i\in \Omega}A_i\right)\right) \subseteq
\alpha\left(\bigwedge\limits_{i\in
\Omega}\mathcal{C}(A_i)\right)=\bigcup\limits_{i\in
\Omega}\alpha(\mathcal{C}(A_i))$.  This shows $a\not\in
\alpha\left(\mathcal{C}\left(\bigvee\limits_{i\in
\Omega}A_i\right)\right)$. Therefore $\bigvee\limits_{i\in
\Omega}A_i\in \mathcal{C}^{[a]}$. The proof is completed. \end{proof}

%\begin{corollary} %If $M$ is completely distributive, then a
%A mapping
%$\mathcal{C} : {\bf2}^X \rightarrow M$   is an $M$-fuzzifying
%convexity if and only if  for each $a\in \alpha(\bot_M)$,
%$\mathcal{C}^{[a]}$ is a convexity.
% \end{corollary}

Now we consider the conditions that a family of
$L$-convexities forms an $(L,M)$-fuzzy convexity. By
Theorem $\ref{11a}$,  we can obtain the following result.

\begin{corollary}% $M$ is completely distributive, and
If $\mathcal{C}$ is an
$(L,M)$-fuzzy convexity,  then
\begin{enumerate}
\item[$(1)$] $\mathcal{C} _{[b]}\subseteq \mathcal{C} _{[a]}$ for any $a,b\in
M\backslash\{\bot_M\}$ with $a\in \beta(b)$.

\item[$(2)$] $\mathcal{C} ^{[b]}\subseteq \mathcal{C} ^{[a]}$ for any $a,b\in
\alpha(\bot_M)$ with $b\in \alpha(a)$.\end{enumerate}
\end{corollary}

\begin{theorem}%Let $M$ be  completely distributive,  and
Let
$\left\{\mathcal{C}^a:\  a\in \alpha(\bot_M)\right\}$ be a family of
$L$-convexities. If $\mathcal{C} ^a=\bigcap\{\mathcal{C} ^b:a\in
\alpha(b)\}$ for all $a\in \alpha(\bot_M)$, then there exists an
$(L,M)$-fuzzy convexity  $\mathcal{C}$ such that
$\mathcal{C}^{[a]}=\mathcal{C}^a$.
\end{theorem}
\begin{proof}The proof  is straightforward  and is omitted.
\end{proof}

\begin{theorem} %Let $M$ be  completely distributive,  and
Let
$\left\{\mathcal{C}_a:\ a\in M\backslash\{\bot_M\}\right\}$ be a family
of $L$-convexities.   If $\mathcal{C} _a=\bigcap\{\mathcal{C} _b:b\in
\beta(a)\}$ for all $a\in M\backslash\{\bot_M\}$, then there exists
an $(L,M)$-fuzzy convexity  $\mathcal{C}$ such that $\mathcal{C}_{[a]}=
\mathcal{C}_a$.
\end{theorem}
\begin{proof} This is straightforward.
\end{proof}

%\begin{proof} Suppose that $\mathcal{C} _a=\bigcap\{\mathcal{C}_b:b\in
%\beta(a)\}$ for all $a\in M\backslash\{\bot_M\}$. Define $\mathcal{C} :
%L^X\rightarrow M$ by
%$$\mathcal{C} (A)=\bigvee\limits_{a\in M}\left(a\wedge \mathcal{C} _a(A)\right)=\bigvee\left\{a\in M:
%A\in \mathcal{C} _a\right\}.$$ By Theorem $\ref{11a}$, we can obtain
%$\mathcal{C}_{[a]}= \mathcal{C}_a$. \end{proof}

%\begin{cor}% Let $M$ be  completely distributive,  and
%Let
%$\left\{\mathcal{C}_a:\ a\in M\backslash\{\bot_M\}\right\}$ be a family
%of convexities. If $\mathcal{C} _a=\bigcap\{\mathcal{C} _b:b\in
%\beta(a)\}$ for all $a\in M\backslash\{\bot_M\}$, then there exists
%an $M$-fuzzifying convexity  $\mathcal{C}$ such that
%$\mathcal{C}_{[a]}= \mathcal{C}_a$.
%\end{cor}

\begin{definition}  Let $\mathcal{C}, \mathcal{D}$ be  $(L,M)$-fuzzy convexities on $X$.
If $\mathcal{C}(A) \leq  \mathcal{D}(A)$ for all $  A\in L^{X}$, i.e., $\mathcal{C} \leq  \mathcal{D}$,   then $\mathcal{C}$ is coarser than $\mathcal{D}$ and $\mathcal{D}$ is finer than $\mathcal{C}$.
\end{definition}

\begin{theorem} \label{cderfv}   Let $\{\mathcal{C}_t: t\in T\}$ be a family of $(L,M)$-fuzzy convexities on $X$.
Then $\bigwedge\limits_{t\in T}\mathcal{C}_t$   is an $(L,M)$-fuzzy
convexity on $X$, where $\bigwedge\limits_{t\in T}\mathcal{C}_t: L^X\rightarrow M$ is defined by
 $\left(\bigwedge\limits_{t\in T}\mathcal{C}_t\right)(A)=
\bigwedge\limits_{t\in T}\mathcal{C}_t (A)$ for each  $A\in L^X$.  Obviously, $\bigwedge\limits_{t\in T}\mathcal{C}_t$ is  coarser than $\mathcal{C}_t$ for all $t\in T$.
\end{theorem}

\begin{proof} This is straightforward.
\end{proof}

%\begin{corollary}\label{321} Let $\{\mathcal{C}_t: t\in T\}$ be a family of $M$-fuzzifying convexities on $X$.
%Then $\bigwedge\limits_{t\in T}\mathcal{C}_t$   is an $M$-fuzzifying
%convexity on $X$, where $\bigwedge\limits_{t\in T}\mathcal{C}_t: \mathbf{2}^X\rightarrow M$ is defined by
% $\left(\bigwedge\limits_{t\in T}\mathcal{C}_t\right)(A)=
%\bigwedge\limits_{t\in T}\mathcal{C}_t (A)$ for each  $A\in \mathbf{2}^X$.
%\end{corollary}

\section{$(L,M)$-fuzzy  convexity preserving functions}

In this section, we shall generalize the notion of convexity preserving functions to
lattice-valued setting.

\begin{definition}  Let $(X,\mathcal{C})$ and $(Y,\mathcal{D})$ be $(L,M)$-fuzzy convex structures.  A function $f: X\rightarrow Y$ is called  an $(L,M)$-fuzzy convexity preserving function
%(an $(L,M)$-fuzzy $\mathbf{CP}$ function)
if
$\mathcal{C}\left(f^{\leftarrow}_L(B)\right)\ge \mathcal{D}(B)$ for all $B\in
L^Y$.
\end{definition}

A $({\bf2},M)$-fuzzy convexity preserving function is  an
$M$-fuzzifying convexity preserving function in \cite{Shi6}.

%It is easy to verify  that  $(L,M)$-fuzzy convex structures and their $(L,M)$-fuzzy convexity preserving functions form a category which is denoted by $\mathbf{LMCS}$ and  $M$-fuzzifying convex structures and their $M$-fuzzifying convexity preserving functions form a category which is denoted by $\mathbf{MYCS}$.

%\begin{definition} Let $(X,\mathcal{C})$ and $(Y,\mathcal{D})$ be $(L,M)$-fuzzy  convex structures.  A function $f: X\rightarrow Y$ is  an $(L,M)$-fuzzy  isomorphism  if
%$f$ is a bijection,  an $(L,M)$-fuzzy convexity preserving function and  an $(L,M)$-fuzzy  convex-to-convex  function.
%\end{definition}

\begin{theorem}   Let $(Y, \mathcal{D})$ be an $(L,M)$-fuzzy convex structure and $f: X\rightarrow
Y$ a surjective function.  Define a mapping
$f^{\leftarrow}_L(\mathcal{D}):L^X\rightarrow M$ by $$\forall A\in L^X,\
f^{\leftarrow}_L(\mathcal{D})(A)=\bigvee\left\{\mathcal{D}(B): f^{\leftarrow}_L(B)=A\right\}.$$
Then $(X,f^{\leftarrow}_L(\mathcal{D}))$ is an $(L,M)$-fuzzy convex structure.
\end{theorem}

\begin{proof}  {\bf (LMC1)} holds from the following equalities:
$$f^{\leftarrow}_L(\mathcal{D})(\chi_{\emptyset})=\bigvee\left\{\mathcal{D}(B):
f^{\leftarrow}_L(B)=\chi_{\emptyset}\right\}=\mathcal{D}(\chi_{\emptyset})=\top_M,$$
and
$$f^{\leftarrow}_L(\mathcal{D})(\chi_{X})=\bigvee\left\{\mathcal{D}(B):
f^{\leftarrow}_L(B)=\chi_{X}\right\}=\mathcal{D}(\chi_{Y})=\top_M.$$

{\bf (LMC2)}
For any nonempty set $\{A_i: i\in \Omega\}\subseteq L^X$, let $a$ be any element in $M$ with the property of  $\bigwedge\limits_{i\in \Omega}f^{\leftarrow}_L(\mathcal{D})
(A_i)\succ a$. For each $i\in \Omega$, $\bigvee\left\{\mathcal{D}(B):
f^{\leftarrow}_L(B)=A_i\right\}=f^{\leftarrow}_L(\mathcal{D})
(A_i)\succ a$.  Then for each  $i\in \Omega$, there exists $B_i\in L^{X}$  such that
$f^{\leftarrow}_L(B_i)=A_i$ and $\mathcal{D}(B_i)\geq a$.  Note that
 $f^{\leftarrow}_L\left(\bigwedge\limits_{i\in \Omega}B_i\right)= \bigwedge \limits_{i\in \Omega}f^{\leftarrow}_L(B_i)=\bigwedge\limits_{i\in \Omega}A_{i}$ and $\mathcal{D}\left(\bigwedge\limits_{i\in \Omega}B_i \right)\geq \bigwedge\limits_{i\in \Omega} \mathcal{D}(B_i) \geq a$. Finally we have
$$f^{\leftarrow}_L(\mathcal{D})\left(\bigwedge\limits_{i\in \Omega}A_i\right)= \bigvee\left\{\mathcal{D}(B):
f^{\leftarrow}_L(B)=\bigwedge\limits_{i\in \Omega}A_i\right\}\geq   \mathcal{D}\left(\bigwedge\limits_{i\in \Omega}B_i\right)\geq a.$$  This implies
$f^{\leftarrow}_L(\mathcal{D})\left(\bigwedge\limits_{i\in \Omega}A_i\right)\geq \bigwedge\limits_{i\in \Omega}f^{\leftarrow}_L(\mathcal{D})
(A_i)$.

{\bf (LMC3)} For any nonempty set $\{A_i: i\in \Omega\}\subseteq L^X$, which is totally ordered by inclusion,   let $a$ be any element in $M$ with the property of  $\bigwedge\limits_{i\in \Omega}f^{\leftarrow}_L(\mathcal{D})
(A_i)\succ a$, that is, $\bigwedge\limits_{i\in \Omega}\bigvee\left\{\mathcal{D}(B):
f^{\leftarrow}_L(B)=A_i\right\}\succ a$. Then $\forall i\in \Omega$, $\bigvee\left\{\mathcal{D}(B):
f^{\leftarrow}_L(B)=A_i\right\}=f^{\leftarrow}_L(\mathcal{D})
(A_i)\succ a$. For each  $i\in \Omega$, there exists $B_i\in L^{X}$  such that
$f^{\leftarrow}_L(B_i)=A_i$ and $\mathcal{D}(B_i)\geq a$.
Since $f$  is  surjective and  $\{A_i: i\in \Omega\}$  is totally ordered by inclusion, we have $\{B_i: i\in \Omega\}$  is totally ordered by inclusion.
Note that
 $f^{\leftarrow}_L\left(\bigvee\limits_{i\in \Omega}B_i\right)= \bigvee \limits_{i\in \Omega}f^{\leftarrow}_L(B_i)=\bigvee\limits_{i\in \Omega}A_{i}$ and $\mathcal{D}\left(\bigvee\limits_{i\in \Omega}B_i \right)\geq \bigwedge\limits_{i\in \Omega} \mathcal{D}(B_i) \geq a$. Finally we have
$$f^{\leftarrow}_L(\mathcal{D})\left(\bigvee\limits_{i\in \Omega}A_i\right)= \bigvee\left\{\mathcal{D}(B):
f^{\leftarrow}_L(B)=\bigvee\limits_{i\in \Omega}A_i\right\}\geq   \mathcal{D}\left(\bigvee\limits_{i\in \Omega}B_i\right)\geq a.$$  This implies
$f^{\leftarrow}_L(\mathcal{D})\left(\bigvee\limits_{i\in \Omega}A_i\right)\geq \bigwedge\limits_{i\in \Omega}f^{\leftarrow}_L(\mathcal{D})
(A_i)$.
\end{proof}

The following theorem gives a characterization of  $(L,M)$-fuzzy
convexity preserving functions.

\begin{theorem}  Let $(X,\mathcal{C})$ and $(Y,\mathcal{D})$ be two $(L,M)$-fuzzy convex structures. A surjective function $f: X\rightarrow Y$ is an $(L,M)$-fuzzy convexity preserving function  if and
only if $ f^{\leftarrow}_L(\mathcal{D})(A) \le \mathcal{C}(A)$ for all $A\in
L^X$.
\end{theorem}

\begin{proof}   Necessity. If $f: X\rightarrow Y$ is an $(L,M)$-fuzzy convexity preserving function,
then $\mathcal{C}\left(f^{\leftarrow}_L(B)\right)\ge \mathcal{D}(B)$ for all
$B\in L^Y$. Hence for all $A\in
L^X$, we have
$$f^{\leftarrow}_L(\mathcal{D})(A)
= \bigvee\left\{\mathcal{D}(B):
f^{\leftarrow}_L(B)=A\right\}
\le  \bigvee\left\{\mathcal{C}\left(f^{\leftarrow}_L(B)\right):
f^{\leftarrow}_L(B)=A\right\}
= \mathcal{C}(A).$$

% $$\begin{array}{lll}
%f^{\leftarrow}_L(\mathcal{D})(A)
%&= &\bigvee\left\{\mathcal{D}(B):
%f^{\leftarrow}_L(B)=A\right\}\\
%&\le  &\bigvee\left\{\mathcal{C}\left(f^{\leftarrow}_L(B)\right):
%f^{\leftarrow}_L(B)=A\right\}\\
%&=& \mathcal{C}(A).\end{array}$$

Sufficiency. If $f^{\leftarrow}_L(\mathcal{D})(A) \le \mathcal{C}(A)$ for all
$A\in L^X$, then
 $$\mathcal{D}(B)\le
\bigvee\left\{\mathcal{D}(G):
f^{\leftarrow}_L(G)=f^{\leftarrow}_L(B)\right\}=
f^{\leftarrow}_L(\mathcal{D})(f^{\leftarrow}_L(B))\le
\mathcal{C}(f^{\leftarrow}_L(B))$$ for all $B\in L^Y$. This shows that $f:
X\rightarrow Y$ is an $(L,M)$-fuzzy convexity preserving function.
\end{proof}

The following   theorems are trivial.

\begin{theorem}\label{246}  \rm If $f: (X,\mathcal{C})\rightarrow (Y,\mathcal{D})$ and $g: (Y,\mathcal{D})\rightarrow
(Z, \mathcal{H})$ are $(L,M)$-fuzzy convexity preserving functions, then $g\circ f:
(X,\mathcal{C})\rightarrow  (Z,\mathcal{H})$ is  an $(L,M)$-fuzzy
convexity preserving function.
\end{theorem}

\begin{theorem}\rm Let $(X,\mathcal{C})$ and $(Y,\mathcal{D})$ be $(L,M)$-fuzzy convex structures. Then  a function $f: (X,\mathcal{C})\rightarrow (Y,\mathcal{D})$ is an  $(L,M)$-fuzzy
convexity preserving function if and only if $f: (X,\mathcal{C}_{[a]})\rightarrow
(Y,\mathcal{D}_{[a]})$ is an $L$-convexity preserving function for any $a\in
M\backslash\{\bot_M\}$.\end{theorem}

\begin{theorem}\rm Let $(X,\mathcal{C})$ and
$(Y,\mathcal{D})$ be $(L,M)$-fuzzy convex structures. Then a function $f:
(X,\mathcal{C})\rightarrow (Y,\mathcal{D})$ is  an   $(L,M)$-fuzzy  convexity preserving function  if and
only if   $f: (X,\mathcal{C}^{[a]})\rightarrow (Y,\mathcal{D}^{[a]})$ is an
$L$-convexity preserving function  for any $a\in \alpha(\bot_M)$.\end{theorem}

%\begin{cor}\rm Let $(X,\mathcal{C})$ and $(Y,\mathcal{D})$ be $M$-fuzzifying convex structures. Then a function $f: (X,\mathcal{C})\rightarrow (Y,\mathcal{D})$ is
% an $M$-fuzzifying convexity preserving function  if and only if $f:
%(X,\mathcal{C}_{[a]})\rightarrow (Y,\mathcal{D}_{[a]})$ is a  convexity preserving function  for any
%$a\in M\backslash\{\bot_M\}$.\end{cor}
%
%\begin{cor}\rm Let  $(X,\mathcal{C})$ and
%$(Y,\mathcal{D})$ be $M$-fuzzifying convex structures. Then a function  $f:
%(X,\mathcal{C})\rightarrow (Y,\mathcal{D})$ is  an $M$-fuzzifying convexity preserving function if and
%only if   $f: (X,\mathcal{C}^{[a]})\rightarrow (Y,\mathcal{D}^{[a]})$ is
%a convexity preserving function for any $a\in \alpha(\bot_M)$.\end{cor}

%\begin{theorem}
%$(L,M)$-fuzzy convex structures and their $(L,M)$-fuzzy convexity preserving functions form a category which is denoted by $\mathbf{LMCS}$.  $M$-fuzzifying convex structures and their $M$-fuzzifying convexity preserving functions form a category which is denoted by $\mathbf{MYCS}$.
%\end{theorem}

\section{ Quotient  $(L,M)$-fuzzy  convex structures }

In this section, the notions of  quotient structures and quotient functions are generalized to lattice-valued  setting.
% Quotient  $(L,M)$-fuzzy  convex structures are the generations of  quotient $I$-convex structures in \cite{Rosa1,Rosa2} and  quotient convex structures.

\begin{theorem}\label{1122} Let $(X,\mathcal{C})$ be an $(L,M)$-fuzzy convex structure and $f: X\rightarrow
Y$ a  surjective function. Define a mapping
$\mathcal{C}_{/f}:L^Y\rightarrow M$ by
 $$\forall B\in L^Y,\
\mathcal{C}_{/f}(B)=\mathcal{C}(f^{\leftarrow}_L(B)).$$
Then $(Y,\mathcal{C}_{/f})$ is an $(L,M)$-fuzzy convex structure and we call  $\mathcal{C}_{/f}$  a quotient $(L,M)$-fuzzy convexity on $Y$  with respect to $f$ and $\mathcal{C}$.  Moreover, it is easy to see that $f$ is an $(L,M)$-fuzzy convexity preserving function from $(X,\mathcal{C})$ to $(Y,\mathcal{C}_{/f})$.
\end{theorem}

\begin{proof}  {\bf (LMC1)} holds from the following equalities:\\
$\mathcal{C}_{/f}(\chi_{\emptyset})
=\mathcal{C}(f^{\leftarrow}_L(\chi_{\emptyset}))=\mathcal{C}(\chi_{\emptyset})=\top_M$
and
$\mathcal{C}_{/f}(\chi_{Y})
=\mathcal{C}(f^{\leftarrow}_L(\chi_{Y}))
=\mathcal{C}(\chi_{X})=\top_M$.

 {\bf (LMC2)} can be shown from the following fact:
 for any nonempty set $\{B_i: i\in \Omega\}\subseteq L^Y$,
$$\begin{array}{lllllll}
\mathcal{C}_{/f}\left(\bigwedge\limits_{i\in \Omega}B_i\right)
&=& \mathcal{C}(f^{\leftarrow}_L(\bigwedge\limits_{i\in \Omega}B_i))
&=& \mathcal{C}(\bigwedge\limits_{i\in \Omega}f^{\leftarrow}_L(B_i))\\
&\ge & \bigwedge\limits_{i\in \Omega}\mathcal{C}(f^{\leftarrow}_L(B_i))
&=&  \bigwedge\limits_{i\in \Omega}\mathcal{C}_{/f}(B_i).
\end{array}$$

{\bf (LMC3)}  If $\{B_i:i\in \Omega\}\subseteq L^Y$  is nonempty and totally ordered by inclusion,  then
$$\begin{array}{llllll}
\mathcal{C}_{/f}\left(\bigvee\limits_{i\in \Omega}B_i\right)
&=& \mathcal{C}(f^{\leftarrow}_L(\bigvee\limits_{i\in \Omega}B_i))
&=& \mathcal{C}(\bigvee\limits_{i\in \Omega}f^{\leftarrow}_L(B_i))\\
&\ge & \bigwedge\limits_{i\in \Omega}\mathcal{C}(f^{\leftarrow}_L(B_i))
&=&  \bigwedge\limits_{i\in \Omega}\mathcal{C}_{/f}(B_i).
\end{array}$$
\end{proof}

\begin{theorem} Let $(X,\mathcal{C})$ be an $(L,M)$-fuzzy convex structure and $f: X\rightarrow  Y$ a  surjective function. Then  $\mathcal{C}_{/f}$ is the finest convexity on $Y$ such that  $f$ is an $(L,M)$-fuzzy convexity preserving function.
\end{theorem}
\begin{proof}
Let $\mathcal{D}$ be  an $(L,M)$-fuzzy convexity on $Y$  such that  $f$ is an $(L,M)$-fuzzy convexity preserving function from $(X,\mathcal{C})$ to $(Y,\mathcal{D})$. Then we have for all $B\in L^{Y}$,  $\mathcal{C}(f^{\leftarrow}_L(B))\geq \mathcal{D}(B)$ and thus
$\mathcal{C}_{/f}(B)=\mathcal{C}(f^{\leftarrow}_L(B))\geq \mathcal{D}(B)$. Therefore $\mathcal{C}_{/f}\geq \mathcal{D}$.
\end{proof}

\begin{definition} Let $(X,\mathcal{C})$ and $(Y,\mathcal{D})$ be $(L,M)$-fuzzy convex structures.  A function $f: X\rightarrow Y$ is called  an $(L,M)$-fuzzy quotient function if  $f$ is surjective  and $\mathcal{D}$ is a quotient $(L,M)$-fuzzy convexity with respect to $f$ and $\mathcal{C}$.
\end{definition}

\begin{theorem}\rm
If $f: (X,\mathcal{C})\rightarrow (Y,\mathcal{D})$ is an $(L,M)$-fuzzy quotient function, then   $g: (Y,\mathcal{D})\rightarrow
(Z, \mathcal{H})$ is an  $(L,M)$-fuzzy convexity preserving function if and only if  $g\circ f: (X,\mathcal{C})\rightarrow  (Z,\mathcal{H})$ is  an $(L,M)$-fuzzy
convexity preserving function.
\end{theorem}
\begin{proof} Since  $f: (X,\mathcal{C})\rightarrow (Y,\mathcal{D})$ is an $(L,M)$-fuzzy quotient function,  we know that   $f$ is surjective  and   $\forall B\in L^{Y}$,
 $\mathcal{D}(B)=\mathcal{C}(f^{\leftarrow}_L(B))$.

Necessity. Since $g: (Y,\mathcal{D})\rightarrow
(Z, \mathcal{H})$ is an  $(L,M)$-fuzzy convexity preserving function,  $\forall A\in L^{Z}$, $\mathcal{D}(g^{\leftarrow}_L(A))\geq\mathcal{H}(A)$. Thus $\forall A\in L^{Z}$, $\mathcal{C}((g\circ f)^{\leftarrow}_L(A))=
 \mathcal{C}(f^{\leftarrow}_L(g^{\leftarrow}_L(A)))
 =\mathcal{D}(g^{\leftarrow}_L(A))\geq\mathcal{H}(A)$.

Sufficiency.  Since  $g\circ f: (X,\mathcal{C})\rightarrow  (Z,\mathcal{H})$ is  an $(L,M)$-fuzzy
convexity preserving function, then $\forall A\in L^{Z}$,
$\mathcal{D}(g^{\leftarrow}_L(A))
=\mathcal{C}(f^{\leftarrow}_L(g^{\leftarrow}_L(A)))
=\mathcal{C}((g\circ f)^{\leftarrow}_L(A))\geq\mathcal{H}(A)$.
\end{proof}

\begin{definition}  Let $(X,\mathcal{C})$ and $(Y,\mathcal{D})$ be $(L,M)$-fuzzy convex structures.  A function $f: X\rightarrow Y$ is called  an $(L,M)$-fuzzy convex-to-convex  function if
$\mathcal{D}\left(f^{\rightarrow}_L(A)\right)\ge \mathcal{C}(A)$ for all $A\in
L^X$.
\end{definition}

\begin{theorem}\rm
If $f: (X,\mathcal{C})\rightarrow (Y,\mathcal{D})$ is a surjective $(L,M)$-fuzzy convexity preserving function and is an  $(L,M)$-fuzzy convex-to-convex function,
then  $\mathcal{D}$ is a  quotient $(L,M)$-fuzzy convexity. Moreover, $f$ is an $(L,M)$-fuzzy quotient function  with respect to $f$ and $\mathcal{C}$.
\end{theorem}
\begin{proof}  Since  $f: (X,\mathcal{C})\rightarrow (Y,\mathcal{D})$ is a surjective $(L,M)$-fuzzy convexity preserving function and an  $(L,M)$-fuzzy convex-to-convex function, we have $\forall B\in L^{Y}$, $\mathcal{C}( f^{\leftarrow}_L(B))\geq\mathcal{D}(B)$
 and $\forall A\in L^{X}$,  $\mathcal{D}(f^{\rightarrow}_L(A))\geq\mathcal{C}(A)$. Since $f$ is surjective, we know for all $B\in L^{Y}$, $f^{\rightarrow}_L( f^{\leftarrow}_L(B))=B$.  Hence $$\mathcal{D}(B)=\mathcal{D}(f^{\rightarrow}_L( f^{\leftarrow}_L(B)))\geq
\mathcal{C}( f^{\leftarrow}_L(B))\geq\mathcal{D}(B).$$  So  $\mathcal{C}( f^{\leftarrow}_L(B))=\mathcal{D}(B)$ for each $B\in L^{Y}$ and then  $\mathcal{D}$ is a  quotient $(L,M)$-fuzzy convexity  with respect to $f$ and $\mathcal{C}$.
\end{proof}

By Theorem  $\ref{1122}$, we can obtain the following result.
\begin{theorem}Let $(X,\mathcal{C})$  be an  $(L,M)$-fuzzy convex structure
 and  $R$  be  an
equivalence  relation  defined  on  $X$.  Let  $X/R$  be
the  usual  quotient  set  and  let  $\pi$ be  the  projection
map  from $X$  to $X/R$.  Define $\mathcal{D}: L^{(X/R)}\rightarrow M$ by $$\forall B\in L^{(X/R)}, \
\mathcal{D}(B)=\mathcal{C}(\pi^{\leftarrow}_L(B)).$$
Then  $\mathcal{D}$  is  an $(L,M)$-fuzzy convexity  on  $X/R$  and  $(X/R,  \mathcal{D})$  is  a quotient  $(L,M)$-fuzzy   convex structure of  $(X,\mathcal{C})$.
\end{theorem}
\begin{corollary}[\cite{Rosa1,Rosa2}]  Let $X$  be  any  set  and  $R$  be  an
equivalence  relation  defined  on  $X$.  Let  $X/R$  be
the  usual  quotient  set  and  let  $\pi$ be  the  projection
map  from $X$  to $X/R$.
If   $(X,\mathfrak{C})$  is  an $I$-convex structure,  then  one  can
define  an $I$-convexity $\mathfrak{D}$  on $X/R$  as  follows:
% Let
%$\mathcal{D}$  be  the  family  of  fuzzy  sets  in $X/R $ defined  by
$\mathfrak{D}=\{B\in I^{(X/R)}: \pi^{\leftarrow}_I(B)\in \mathfrak{C}\}$.  Then $\mathfrak{D}$  is  an $I$-convexity  on  $X/R$  and  $(X/R,  \mathfrak{D})$  is  called  the
quotient  $I$-convex structure .
\end{corollary}

\section{Substructures and   products of  $(L,M)$-fuzzy  convex structures}

In this section,
 we give  substructures and  products  of  $(L,M)$-fuzzy  convex structures and  discuss  some of their fundamental  properties.

\begin{lemma}\label{123qwe}
Let $(X,\mathfrak{C})$ be an  $L$-convex structure and $\emptyset\neq Y \subseteq X$.
For $A\in \mathfrak{C}$, $co(A|Y)|Y=A|Y$.
\end{lemma}
\begin{proof}On the one hand, it is obvious that  $A|Y \subseteq co(A|Y)$. Then $A|Y=(A|Y)|Y \subseteq co(A|Y)|Y$. On the other hand, $A|Y\subseteq A$. Hence $co(A|Y)\subseteq co(A)=A$ and  then $co(A|Y)|Y\subseteq co(A)|Y=A|Y$. Therefore, $co(A|Y)|Y=A|Y$.
\end{proof}

\begin{theorem}  Let $(X,\mathcal{C})$ be an $(L,M)$-fuzzy convex structure, $\emptyset \neq Y\subseteq X$.  Then $(Y,\mathcal{C}|Y)$ is an $(L,M)$-fuzzy convex structure on $Y$, where  $\forall A\in L^{Y}$,  $(\mathcal{C}|Y)(A)=\bigvee\{\mathcal{C}(B):B\in L^{X},\  B|Y =A \}$.  We call $(Y,\mathcal{C}|Y)$  an
$(L,M)$-fuzzy substructure of $(X,\mathcal{C})$.
\end{theorem}

\begin{proof}
(1)Clearly, $(\mathcal{C}|Y)(\chi_{\emptyset})=(\mathcal{C}|Y)(\chi_{X})=\top_M$.

(2) For any nonempty set $\{A_{i} : i\in \Omega \}\subseteq L^{Y}$, we have
$$\begin{array}{lll}
&&\bigwedge\limits_{i\in \Omega}(\mathcal{C}|Y)(A_{i})\\
&=& \bigwedge\limits_{i\in \Omega}\bigvee\{\mathcal{C}(B):B\in L^{X},\  B|Y =A_{i}\}\\
&=& \bigvee\limits_{f  \in \Pi_{i\in \Omega}H_{i}}\bigwedge\limits_{i\in \Omega}\mathcal{C}(f(i))\\
&\leq& \bigvee\limits_{f  \in \Pi_{i\in \Omega}H_{i}}\mathcal{C}\left(\bigwedge\limits_{i\in \Omega}f(i)\right),
\end{array}$$
where $H_{i}=\{B: B\in L^{X},\  B|Y =A_{i} \}$ $(i\in \Omega)$. Since
$\left(\bigwedge\limits_{i\in \Omega}f(i)\right)|Y=\bigwedge\limits_{i\in \Omega}(f(i)|Y)=\bigwedge\limits_{i\in \Omega}A_{i}$, we have $(\mathcal{C}|Y)  \left( \bigwedge\limits_{i\in \Omega}A_{i} \right)  \geq \bigwedge\limits_{i\in \Omega}(\mathcal{C}|Y)(A_{i})$.

(3) For any $\{A_{i} : i\in \Omega \}\subseteq L^{Y}$, which  is nonempty and totally ordered by inclusion, we have
$$\begin{array}{lll}
\bigwedge\limits_{i\in \Omega}(\mathcal{C}|Y)(A_{i})
&=& \bigwedge\limits_{i\in \Omega}\bigvee\{\mathcal{C}(B):B\in L^{X},\  B|Y =A_{i}\}\\
&=& \bigvee\limits_{f  \in \Pi_{i\in \Omega}\mu_{i}}\bigwedge\limits_{i\in \Omega}\mathcal{C}(f(i))\\
&\leq& \bigvee\limits_{f  \in \Pi_{i\in \Omega}\mu_{i}}\mathcal{C}\left(\bigvee\limits_{i\in \Omega}f(i)\right),
\end{array}$$
where $\mu_{i}=\{B: B\in L^{X},\  B|Y =A_{i} \}$ $(i\in \Omega)$. Since
$\left(\bigvee\limits_{t\in T}f(i)\right)|Y=\bigvee\limits_{i\in \Omega}(f(i)|Y)=\bigvee\limits_{i\in \Omega}A_{i}$, we have $(\mathcal{C}|Y)  \left( \bigvee\limits_{i\in \Omega}A_{i} \right)  \geq \bigwedge\limits_{i\in \Omega}(\mathcal{C}|Y)(A_{i})$.

(3) For any set $\{A_{i} : i\in \Omega \}\subseteq L^{Y}\subseteq L^{X}$, which  is nonempty and totally ordered by inclusion,  let $a$ be any element in $M\backslash\{\bot\}$ with the property of  $\bigwedge\limits_{i\in \Omega}(\mathcal{C}|Y)(A_{i})
 \succ a$, that is, $\bigwedge\limits_{i\in \Omega}\bigvee\{\mathcal{C}(B):B\in L^{X},\  B| Y =A_{i}\}\succ a$.
Then for  each $i\in \Omega$,  there exists $B_i\in L^{X}$  such that
$B_i| Y=A_i$ and $\mathcal{C}(B_i)\geq a$, i.e., $B_i\in \mathcal{C}_{[a]}$.
By Theorem $\ref{cdax2}$,  for each $a \in M\backslash\{\bot\}$,  $(X, \mathcal{C}_{[a]})$ is a convex structure.
 Let $co_{a}$ denote the  hull operator  of $(X, \mathcal{C}_{[a]})$ for each $a \in M\backslash\{\bot\}$.  Then $co_{a}(A_{i})\in \mathcal{C}_{[a]}$ for all $i\in \Omega$.
 Since $\{A_{i} : i\in \Omega \}\subseteq L^{Y}$  is nonempty and totally ordered by inclusion,  $\{ co_{a}(A_{i}) : i\in \Omega \}$  is nonempty and totally ordered by inclusion. Hence, $\bigvee\limits_{i\in \Omega}co_{a}(A_{i})\in \mathcal{C}_{[a]}$, that is,  $\mathcal{C}(\bigvee\limits_{i\in \Omega}co_{a}(A_{i}))\geq a$.  By Lemma $\ref{123qwe}$,
$$\begin{array}{llllll}
(\bigvee\limits_{i\in \Omega}co_{a}(A_{i}))| Y
&=&(\bigvee\limits_{i\in \Omega}co_{a}(B_i| Y))| Y\\
&=&\bigvee\limits_{i\in \Omega}co_{a}(B_i| Y)| Y\\
&=&\bigvee\limits_{i\in \Omega}(B_i| Y)\\
&=&\bigvee\limits_{i\in \Omega}A_i.
\end{array}$$
  So we have $(\mathcal{C}|Y)  \left(\bigvee\limits_{i\in \Omega}A_{i} \right)\geq a$. This implies $(\mathcal{C}|Y)  \left(\bigvee\limits_{i\in \Omega}A_{i} \right)  \geq \bigwedge\limits_{i\in \Omega}(\mathcal{C}|Y)(A_{i})$.
\end{proof}

\begin{corollary} [\cite{Rosa1,Rosa2}]  Let $(X,\mathfrak{C})$ be an  $I$-convex structure,  $\emptyset \neq Y\subseteq X$.  Then  an $I$-convexity  $\mathfrak{C}|Y$  on $Y$  is  given  by  the  fuzzy  sets  of  the  form   $\{  B|Y:   B\in \mathfrak{C}\}$.   The  pair  $(Y, \mathfrak{C}|Y)$  is  an  $I$-convex  substructure of $(X,\mathfrak{C})$.
\end{corollary}

By Theorem $\ref{cderfv}$, we can give the following definition:
\begin{definition}\label{ac}
 Let $\varphi: L^{X}\rightarrow M$ be a mapping.  The $(L,M)$-fuzzy convex structure  $(X,\mathcal{C})$ generated by $\varphi$ is given by
  $$\forall A\in  L^{X},\    \mathcal{C}(A)=\bigwedge \{\mathcal{D}(A): \varphi\leq  \mathcal{D}\in \mathfrak{\mathfrak{H}}  \},$$ where  $\mathfrak{\mathfrak{H}}$ denotes all the $(L,M)$-fuzzy convexities on $X$.  Then  $\varphi$ is called a subbase of   the $(L,M)$-fuzzy convexity  $\mathcal{C}$.  Alternatively, we say that  $\varphi$ generates the convexity $\mathcal{C}$.
\end{definition}

Based on  Definition $\ref{ac}$, we can define the product of  $(L,M)$-fuzzy  convex structures as follows:

\begin{definition}
Let $\{(X_{t} , \mathcal{C}_{t} )\}_{t\in T}$  be a family of $(L,M)$-fuzzy convex structures. Let $X$ be the product of the sets of $X_{t}$  for $t \in T$,  and let $\pi_{t}: X\rightarrow X_{t}$ denote the projection for each $t \in T$. Define a maping $\varphi:  L^{X}\rightarrow M$ by
$\varphi(A)=
\bigvee\limits_{t\in T}\bigvee\limits_{(\pi_{t})^{\leftarrow}_L(B)=A}\mathcal{C}_{t}(B)$ for each $A\in L^{X}$.
Then the product convexity  $\mathcal{C}$ of $X$  is the one generated by the subbase $\varphi$.
The resulting $(L,M)$-fuzzy convex structure $(X, \mathcal{C})$ is called the  product  of $\{(X_{t} , \mathcal{C}_{t} )\}_{t\in T}$  and is donated by $\prod\limits_{t\in T}(X_{t} , \mathcal{C}_{t})$.
\end{definition}

When $L=[0,1]$ and $M=\mathbf{2}$,  we can obtain the following definition.

\begin{definition}[\cite{Rosa1,Rosa2}]
Let $\{(X_{t} , \mathfrak{C}_{t} )\}_{t\in T}$  be a family of $I$-convex structures. Let $X$ be the product of the sets of $X_{t}$  for $t \in T$,  and let $\pi_{t}: X\rightarrow X_{t}$ denote the projection for each $t \in T$.
Then  X  can  be  equipped  with
the   $I$-convexity $\mathfrak{C}$  generated  by  the  convex
fuzzy sets of the  form  $\{(\pi_{t})^{\leftarrow}_I(B): B\in \mathfrak{C}_{t}, t\in T\}$.
Then  $\mathfrak{C}$ is called  the  product $I$-convexity  for
$X$  and  $(X, \mathfrak{C})$ is  called  the  product  $I$-convex structure.
\end{definition}

\begin{theorem}
Let $(X , \mathcal{C})$  be the product  of  $\{(X_{t} , \mathcal{C}_{t} )\}_{t\in T}$. Then $\forall t \in T$,   $\pi_{t}: X\rightarrow X_{t}$ is an
$(L,M)$-fuzzy convexity preserving function.  Moreover, $\mathcal{C}$ is the coarsest $(L,M)$-fuzzy convex structure  such that  $\{\pi_{t}: t\in T \}$  are $(L,M)$-fuzzy convexity preserving functions.
\end{theorem}
\begin{proof}  Let $t_{0}\in T$.  $\forall B\in L^{X_{t_{0}}}$,  by

%$$\begin{array}{llllll}
%\mathcal{C}((\pi_{t_{0}})^{\leftarrow}_L(B))\geq \varphi((\pi_{t_{0}})^{\leftarrow}_L(B))=
%\bigvee\limits_{t\in T}\bigvee\limits_{(\pi_{t})^{\leftarrow}_L(B)
%=(\pi_{t_{0}})^{\leftarrow}_L(B)}\mathcal{C}_{t}(B)\geq \mathcal{C}_{t_{0}}(B),
%\end{array}$$

$$\begin{array}{llllll}
\mathcal{C}_{t_{0}}(B)
&\leq&  \bigvee\limits_{t\in T}\bigvee\limits_{(\pi_{t})^{\leftarrow}_L(B)
=(\pi_{t_{0}})^{\leftarrow}_L(B)}\mathcal{C}_{t}(B)\\
&=& \varphi((\pi_{t_{0}})^{\leftarrow}_L(B))\\
&\leq& \mathcal{C}((\pi_{t_{0}})^{\leftarrow}_L(B)),
\end{array}$$
it implies that $\pi_{t_{0}}: X\rightarrow X_{t_{0}}$ is an
$(L,M)$-fuzzy convexity preserving function.  By the arbitrariness of  $t_{0}$,  we know   $\forall t \in T$,  $\pi_{t}: X\rightarrow X_{t}$ is an
$(L,M)$-fuzzy convexity preserving function.  If there is an $(L,M)$-fuzzy convex structure $\mathcal{D}$ on $X$ such that $\forall t\in T$, $\pi_{t} : X\rightarrow  X_{t}$  is an  $(L,M)$-fuzzy convexity preserving function,
then we need to prove $\mathcal{D}\geq \mathcal{C}$.  $\forall B\in L^{X}$, $t\in T$, if $(\pi_{t})^{\leftarrow}_L(G)=B$, $\mathcal{D}(B)= \mathcal{D}((\pi_{t})^{\leftarrow}_L(G))\geq  \mathcal{C}_{t}(G)$.
Note that  $\varphi(B)=
\bigvee\limits_{t\in T}\bigvee\limits_{(\pi_{t})^{\leftarrow}_L(G)=B}\mathcal{C}_{t}(G)$. We have $\mathcal{D}(B)\geq \varphi(B)$ for all $B\in L^{X}$. Hence  $\mathcal{D}\geq \mathcal{C}$.
\end{proof}

\section{Relation between  $\mathbf{MYCS}$ and   $\mathbf{LMCS}$}

%It is easy to verify  that  $(L,M)$-fuzzy convex structures and their $(L,M)$-fuzzy convexity preserving functions form a category which is denoted by $\mathbf{LMCS}$ and  $M$-fuzzifying convex structures and their $M$-fuzzifying convexity preserving functions form a category which is denoted by $\mathbf{MYCS}$.
%
%In this section, we  create a functor $\omega$ from $\mathbf{MYCS}$
% to $\mathbf{LMCS}$ and show that there exists an adjunction between  $\mathbf{MYCS}$ and   $\mathbf{LMCS}$.    We always suppose  that  $\beta(a\wedge b) = \beta(a)\cap\beta(b)$ for any $a, b \in L$ in this section.

In this section,  we discuss the   relation   between   $(L,M)$-fuzzy convex structures and $M$-fuzzifying convex structures  from a categorical aspect.  $(L,M)$-fuzzy convex structures and their $(L,M)$-fuzzy convexity preserving functions form a category which is denoted by $\mathbf{LMCS}$ and  $M$-fuzzifying convex structures and their $M$-fuzzifying convexity preserving functions form a category which is denoted by $\mathbf{MYCS}$.  Moreover,   we  create a functor $\omega$ from $\mathbf{MYCS}$
 to $\mathbf{LMCS}$ and show that there exists an adjunction between  $\mathbf{MYCS}$ and   $\mathbf{LMCS}$.    We always suppose  that  $\beta(a\wedge b) = \beta(a)\cap\beta(b)$ for any $a, b \in L$ in this section.

\begin{lemma} \label{qwas} If $\beta(a\wedge b) = \beta(a)\cap\beta(b)$ for any $a, b \in L$, then  for $A\in L^X$,  $\{A_{[c]}: b\in \beta(c)\}$  is  up-directed.
\end{lemma}
\begin{proof} Let $A_{[c_{1}]}, A_{[c_{2}]}\in \{A_{[c]}: b\in \beta(c)\}$. Then
$b\in \beta(c_{1})$  and $b\in \beta(c_{2})$. We have $b\in  \beta(c_{1})\bigcap \beta(c_{2})=\beta(c_{1}\wedge c_{2})$.  Hence $A_{[c_{1}\wedge c_{2}]}\in  \{A_{[c]}: b\in \beta(c)\}$. Moreover, $A_{[c_{1}]}, A_{[c_{2}]}\subseteq A_{[c_{1}\wedge c_{2}]}$. Therefore, $\{A_{[c]}: b\in \beta(c)\}$  is  up-directed.
\end{proof}

% $\{c\in L: b\in \beta(c) \}$ is a down-directed set.

\begin{theorem} Let $(X,\mathscr{C})$ be an $M$-fuzzifying convex structure. Define a mapping $\omega(\mathscr{C}): L^X\rightarrow M$ by
$$ \forall A\in L^X,\    \omega(\mathscr{C})(A)=\bigwedge\limits_{a\in L}\mathscr{C}(A_{[a]}).$$
Then $\omega(\mathscr{C})$  is  an  $(L,M)$-fuzzy convexity.
%which is said to be the $(L,M)$-fuzzy convexity generated by
%$\mathscr{C}$.
\end{theorem}

\begin{proof}  {\bf (LMC1)} Obviously, $\omega(\mathscr{C})(\chi_X)=\omega(\mathscr{C})(\chi_\emptyset)=\top_{M}$.

{\bf (LMC2)} For any nonempty set  $\{A_i: i\in \Omega\}\subseteq L^X$,   we have
$$\begin{array}{lllllll}
\omega(\mathscr{C})(\bigwedge\limits_{i\in \Omega}A_i)
&=&  \bigwedge\limits_{a\in L}\mathscr{C}((\bigwedge\limits_{i\in \Omega}A_i)_{[a]})
&=&  \bigwedge\limits_{a\in L}\mathscr{C}(\bigwedge\limits_{i\in \Omega}(A_i)_{[a]})\\
&\geq &  \bigwedge\limits_{a\in L}\bigwedge\limits_{i\in \Omega}\mathscr{C}((A_i)_{[a]})
&=&\bigwedge\limits_{i\in \Omega} \bigwedge\limits_{a\in L}\mathscr{C}((A_i)_{[a]})\\
&=& \bigwedge\limits_{i\in \Omega} \omega(\mathscr{C})(A_i).\end{array}$$

{\bf(LMC3)}
For any set  $\{A_i: i\in \Omega\}\subseteq L^X$, which is nonempty and totally ordered by inclusion,   we need to prove that $\omega(\mathscr{C})\left(\bigvee\limits_{i\in \Omega}A_i\right) \geq
\bigwedge\limits_{i\in \Omega}\omega(\mathscr{C})(A_i)$, that is,
$\bigwedge\limits_{a\in L}\mathscr{C}((\bigvee\limits_{i\in \Omega}A_i)_{[a]}) \geq \bigwedge\limits_{i\in \Omega} \bigwedge\limits_{a\in L}\mathscr{C}((A_i)_{[a]})$.  Let $h\in M\backslash\{\bot_M\}$  and $\bigwedge\limits_{i\in \Omega} \bigwedge\limits_{a\in L}\mathscr{C}((A_i)_{[a]})\geq h$. Then we have  for any $i\in \Omega$ and for any $a\in L$, $\mathscr{C}((A_i)_{[a]})\geq h$, i.e., $(A_i)_{[a]}\in  \mathscr{C}_{[h]}$.  Since $(X,\mathscr{C})$ is an $M$-fuzzifying convex structure,  by Theorem $\ref{cdax2}$, for each $h\in M\backslash\{\bot_M\}$, $(X,\mathscr{C}_{[h]})$ is a  convex structure. By Theorems $\ref{12q}$ and  $\ref{12w}$,  we know that
$$(\bigvee\limits_{i\in \Omega}A_i)_{[a]}
 =\bigcap\limits_{b\in \beta(a)}(\bigvee\limits_{i\in \Omega}A_i)_{(b)}
 =\bigcap\limits_{b\in \beta(a)}\bigcup\limits_{i\in \Omega}(A_i)_{(b)}
=\bigcap\limits_{b\in \beta(a)}\bigcup\limits_{i\in \Omega}\bigcup\limits_{b\in \beta(c)}(A_i)_{[c]}.$$
%$$\begin{array}{llll}
% (\bigvee\limits_{i\in \Omega}A_i)_{[a]}
% &=&\bigcap\limits_{b\in \beta(a)}(\bigvee\limits_{i\in \Omega}A_i)_{(b)}\\
% &=&\bigcap\limits_{b\in \beta(a)}\bigcup\limits_{i\in \Omega}(A_i)_{(b)}\\
%&=&\bigcap\limits_{b\in \beta(a)}\bigcup\limits_{i\in \Omega}\bigcup\limits_{b\in \beta(c)}(A_i)_{[c]}.
%\end{array}$$
By Lemma $\ref{qwas}$,
for each $b\in  \beta(a)$  and for each $i\in \Omega$,
%since  $\{c\in L: b\in \beta(c) \}$ is a down-directed set,
$\{(A_i)_{[c]}: b\in \beta(c)\} \subseteq \mathscr{C}_{[h]}$ is  up-directed.
Then by Definition  $\ref{qwerd}$,  we have  $\bigcup\limits_{b\in \beta(c)}(A_i)_{[c]}\in \mathscr{C}_{[h]}$.
Let $B_{i}=\bigcup\limits_{b\in \beta(c)}(A_i)_{[c]}$ for each $i\in \Omega$.  Since $\{A_i: i\in \Omega\}$ is totally ordered, we obtain $\{B_i: i\in \Omega\}$ is totally ordered. Then $\bigcup\limits_{i\in \Omega}\bigcup\limits_{b\in \beta(c)}
(A_i)_{[c]}\in \mathscr{C}_{[h]}$.    Therefore  $(\bigvee\limits_{i\in \Omega}A_i)_{[a]}=\bigcap\limits_{b\in  \beta(a)}\bigcup\limits_{i\in \Omega}\bigcup\limits_{b\in \beta(c)}(A_i)_{[c]}\in  \mathscr{C}_{[h]}$. Hence $\omega(\mathscr{C})\left(\bigvee\limits_{i\in \Omega}A_i\right)\geq h$.  By the arbitrariness of $h$, we have $\omega(\mathscr{C})\left(\bigvee\limits_{i\in \Omega}A_i\right) \geq
\bigwedge\limits_{i\in \Omega}\omega(\mathscr{C})(A_i)$.
\end{proof}

%\begin{cor} Let $(X,\mathcal{C})$ be a convex structure. Define a subset $\omega(\mathcal{C})
%\subseteq L^X$
%%($can be viewed as a mapping $\omega(\mathcal{C}):
%%L^X\rightarrow \bf2$$)$
% by
%$$\omega(\mathcal{C})=\{A\in L^X: \forall a\in L,  A_{[a]} \in \mathcal{C}\}.$$
% Then $\omega(\mathcal{C})$ is an  $L$-convexity.
% \end{cor}

\begin{theorem}\rm Let $(X,\mathscr{C})$ and $(Y,\mathscr{D})$ be two $M$-fuzzifying convex structures and
$f: X\rightarrow Y$   a function. Then $f: (X,\mathscr{C})\rightarrow
(Y,\mathscr{D})$ is an  $M$-fuzzifying convexity preserving function if and only if $f:
(X,\omega(\mathscr{C}))\rightarrow (Y,\omega(\mathscr{D}))$ is an  $(L,M)$-fuzzy
convexity preserving function.
\end{theorem}

\begin{proof}  Necessity. Suppose that $f: (X,\mathscr{C})\rightarrow (Y,\mathscr{D})$
is an $M$-fuzzifying convexity preserving function. Then $\mathscr{C}\left(f^{-1}(A)\right)\ge
\mathscr{D}(A)$ for any $A\in {\bf2}^Y$. In order to prove that $f:
(X,\omega(\mathscr{C}))\rightarrow (Y,\omega(\mathscr{D}))$ is an $(L,M)$-fuzzy
convexity preserving function, we need to prove
$\omega(\mathscr{C})\left(f^{\leftarrow}_L(A)\right)\ge \omega(\mathscr{D})(A)$
for any $A\in L^Y$.
For any $A\in L^Y$ and for any $a\in
L$, we have $f^{\leftarrow}_L(A)_{[a]}=f^{-1}(A_{[a]})$.
%$\forall y\in Y$, $y\in f^{\leftarrow}_L(A)_{[a]} \Leftrightarrow f^{\leftarrow}_L(A)(y)\geq a \Leftrightarrow  A(f(y))\geq a \Leftrightarrow f(y)\in A_{[a]} \Leftrightarrow y\in f^{-1}(A_{[a]})$.

In fact, for any $A\in L^Y$, by
$$\begin{array}{llllllll}
\omega(\mathscr{C})\left(f^{\leftarrow}_L(A)\right)
&=&\bigwedge\limits_{a\in L}\mathscr{C}(f^{\leftarrow}_L(A)_{[a]})
&=&\bigwedge\limits_{a\in L}\mathscr{C}(f^{-1}(A_{[a]}))\\
&\geq &\bigwedge\limits_{a\in L}\mathscr{D}(A_{[a]})
&=&\omega(\mathscr{D})(A),
\end{array}$$
we can prove the necessity.

Sufficiency. Suppose that $f: (X,\omega(\mathscr{C}))\rightarrow
(Y,\omega(\mathscr{D}))$ is  an $(L,M)$-fuzzy convexity preserving function. Then
$\omega(\mathscr{C})\left(f_L^{\leftarrow}(A)\right)\ge \omega(\mathscr{D})(A)$
for any $A\in L^Y$. In particular, it follows that
$\omega(\mathscr{C})\left(f_L^{\leftarrow}(A)\right)\ge \omega(\mathscr{D})(A)$
for any $A\in {\bf2}^Y$. In order to prove that $f:
(X,\mathscr{C})\rightarrow (Y,\mathscr{D})$ is an $M$-fuzzifying convexity preserving function, we
need to prove $\mathscr{C}\left(f^{-1}(A)\right)\ge \mathscr{D}(A)$ for any
$A\in {\bf2}^Y$. In fact, for any $A\in {\bf2}^Y$,
%and for any $a\in L$,
%$f^{\leftarrow}_L(A)_{[a]}=f^{-1}(A_{[a]})=f^{-1}(A)$. However,
we have
$$\begin{array}{llllllll}
\mathscr{C}\left(f^{-1}(A)\right)
&=& \mathscr{C}\left(f^{-1}(A_{[a]})\right)
&=&\mathscr{C}\left(f^{\leftarrow}_L(A)_{[a]}\right)\\
&=& \bigwedge\limits_{a\in L}\mathscr{C}(f^{\leftarrow}_L(A)_{[a]})
&=&\omega(\mathscr{C})\left(f^{\leftarrow}_L(A)\right)\\
&\ge& \omega(\mathscr{D})\left(A\right)
&=& \bigwedge\limits_{a\in L}\mathscr{D}(A_{[a]})\\
&=&\mathscr{D}(A).
\end{array}$$ This shows that $f:
(X,\mathscr{C})\rightarrow (Y,\mathscr{D})$ is an $M$-fuzzifying convexity preserving function.
\end{proof}

\begin{theorem}
Suppose that $(X, \mathcal{C})$ is  an $(L,M)$-fuzzy convex  structure.
We can obtain an $M$-fuzzifying convex structure $\iota(\mathcal{C})$ on $X$ generated by the
subbase
$\varphi_{\mathcal{C}}(U): 2^{X}\rightarrow M$ defined as follows:
$$\forall U\in \mathbf{2}^{X},\
\varphi_{\mathcal{C}}(U)=\bigvee\limits_{a\in L}\bigvee\{ \mathcal{C}(B):
 B\in L^{X}, B_{[a]}=U\}.$$
Then  $\iota \circ \omega = id$  and   $\omega \circ \iota \geq  id$.
% Moreover,  $id_{X} : (X, \mathcal{C})\rightarrow (X, \iota \circ \omega(\mathscr{C}))$ is an $I$-fuzzifying convexity preserving.
\end{theorem}
\begin{proof}
We observe that for every $M$-fuzzifying convex structure $\mathscr{C}$ on $X$ the relation  $\varphi_{\omega(\mathscr{C})}(U) \geq \mathscr{C}(U)$ holds for all $U \in \mathbf{2}^{X}$. In fact, it could be showed by
$ \forall U \in \mathbf{2}^{X}$,
$$\begin{array}{lllll}
\varphi_{\omega(\mathscr{C})}(U)
&=&\bigvee\limits_{a\in L}\bigvee\{ \omega(\mathscr{C})(B):
 B\in L^{X}, B_{[a]}=U\}\\
 &\geq &  \omega(\mathscr{C})(U)\\
 &=&\bigwedge\limits_{a\in L}\mathscr{C}(U_{[a]})=\mathscr{C}(U).
\end{array}$$
Thus,  $\iota(\omega(\mathscr{C}))\geq \mathscr{C}$, i.e.,  $\iota \circ \omega \geq id$.

 Conversely,
 let $U \in \mathbf{2}^{X}$ and take any $a\in L$. Then for each
$B \in L^{X}$  with  $B_{[a]} = U$,
$$\omega(\mathscr{C})(B)=\bigwedge\limits_{b\in L}\mathscr{C}(B_{[b]})\leq \mathscr{C}(U).$$
Hence,
$\varphi_{\omega(\mathscr{C})}(U)
=\bigvee\limits_{a\in L}\bigvee\{ \omega(\mathscr{C})(B):
 B\in L^{X}, B_{[a]}=U\}\leq \mathscr{C}(U)$.
It means that $\iota(\omega(\mathscr{C}))\leq \mathscr{C}$, i.e., $\iota \circ \omega \leq id$. Finally, we obtain $\iota \circ \omega = id$ by all proofs above.

Let $(X, \mathcal{C})$ be  an $(L,M)$-fuzzy convex  structure. Then $$\varphi_{\mathcal{C}}(U)
=\bigvee\limits_{a\in L}\bigvee\{\mathcal{C}(B):
 B\in L^{X}, B_{[a]}=U\}$$ for all $U \in \mathbf{2}^{X}$ and  $\iota(\mathcal{C})=\bigwedge\{\mathcal{D}: \varphi_{\mathcal{C}}\leq \mathcal{D}\in\mathfrak{\mathfrak{H}}\}$,  where  $\mathfrak{\mathfrak{H}}$ denotes all the $M$-fuzzifying  convexities on $X$.
$\forall A\in L^{X}$, by
 $$\begin{array}{lllllll}
(\omega \circ \iota (\mathcal{C}))(A)
&=&\bigwedge\limits_{a\in L}\iota (\mathcal{C})(A_{[a]})
\geq   \bigwedge\limits_{a\in L}\varphi_{\mathcal{C}}(A_{[a]})\\
&=&\bigwedge\limits_{a\in L} \bigvee\limits_{b\in L}\bigvee\{\mathcal{C}(B):
 B\in L^{X}, B_{[b]}=A_{[a]}\}\\
& \geq&  \mathcal{C}(A),
\end{array}$$
we have  $\omega \circ \iota (\mathcal{C})\geq \mathcal{C}$, i.e., $\omega \circ \iota \geq id$.
\end{proof}

\begin{theorem} Let $(X, \mathscr{C})$ be  an $M$-fuzzifying  convex  structure,  $(X, \mathcal{D})$ be  an $(L,M)$-fuzzy convex  structure  and  $f : (X, \mathscr{C}) \rightarrow  (Y, \iota (\mathcal{D}))$ be  an $M$-fuzzifying  convexity preserving function.
Then $f : (X, \omega(\mathscr{C})) \rightarrow (Y, \mathcal{D})$  is an $(L,M)$-fuzzy  convexity preserving function.
\end{theorem}
\begin{proof}
Since $f : (X, \mathscr{C}) \rightarrow  (Y,  \iota (\mathcal{D}))$ is an $M$-fuzzifying  convexity preserving function,   $\forall A\in \mathbf{2}^{Y}$,   $\mathscr{C}(f^{-1}(A))\geq \iota (\mathcal{D})(A)$.
 $\forall B\in L^{Y}$, by
$$\begin{array}{lllllll}
\mathcal{D}(B) &\leq &  \omega\circ \iota (\mathcal{D})(B)
&=&  \omega(\iota (\mathcal{D}))(B)\\
&=&\bigwedge\limits_{a\in L}\iota (\mathcal{D})(B_{[a]})
&\leq &\bigwedge\limits_{a\in L}\mathcal{C} (f^{-1}(B_{[a]}))\\
&=&\bigwedge\limits_{a\in L}\mathscr{C}(f^{\leftarrow}_{L}(B)_{[a]})
&=&\omega(\mathscr{C})(f^{\leftarrow}_{L}(B)),
\end{array}$$
we obtain  $f : (X, \omega(\mathscr{C})) \rightarrow (Y, \mathcal{D})$  is  an $(L,M)$-fuzzy  convexity preserving function.

  %Since $f : (X, \mathcal{C}) \rightarrow  (Y,  \iota (\mathcal{D}))$ is an $I$-fuzzifying  convexity preserving,  $\forall B\in I^{Y}$,  $\mathcal{C}(f^{\leftarrow}_{I}(B))=\mathcal{C}( B\circ f)
%\geq \iota (\mathcal{D})(B)$.  Because $\omega$ has order-preserving property,
%$\omega(\mathcal{C}(\bigcirc \circ f)) \geq \omega(\iota (\mathcal{D}))\geq \mathcal{D}$. Moreover, $\forall B\in I^{Y}$,
%$$\begin{array}{lllll}
%\mathcal{D}(B) &\leq &  \omega(\mathcal{C}(\bigcirc\circ f))(B)\\
%&=&\bigwedge\limits_{a\in (0,1]}\mathcal{C}(B_{[a]}\circ f)\\
%&=&\bigwedge\limits_{a\in (0,1]}\mathcal{C} (f^{\leftarrow}_{I}(B))_{[a]}\\
%&=&\omega(\mathcal{C})(f^{\leftarrow}_{I}(B)).
%\end{array}$$
%Therefore,  $f : (X, \omega(\mathcal{C})) \rightarrow (Y, \mathcal{D})$  is  an $(I,I)$-fuzzy  convexity preserving.
\end{proof}
%$\omega(\mathcal{C}(f^{\leftarrow}_{I}(B)))\geq \omega((\iota (\mathcal{D})(B))$.

Based on  the above results, we finally obtain the following theorem.
\begin{theorem}\label{mnb}
There exists an adjunction between  $\mathbf{MYCS}$ and  $\mathbf{LMCS}$.
\end{theorem}

{\section{Conclusion}}
%A certain disadvantage of all the approaches considered in \cite{{Maruyama},{Rosa1},{Rosa2}}   is some inconsistency in the use of the idea of fuzziness.
%An $I(L)$-convexity is an ordinary subset of the family of all fuzzy ($L$-fuzzy) subsets of a given set $X$.  In this paper,  another  more consistent approach to the use of ideas of ``fuzzy mathematics" in classical convex structures is developed,  that is,  the  notion of $(L,M)$-fuzzy convex structures is
%introduced.

In this paper,  combining $L$-convex structures \cite{{Maruyama},{Rosa1},{Rosa2}} and $M$-fuzzifying convex structures \cite{Shi6} and   based on  complete  distributive  lattices $L$ and $M$,
 we present
 a more general approach to the fuzzification of
convex structures.
It is a generalization of  $L$-convex structures and $M$-fuzzifying convex structures.    Under the framework   of  $(L,M)$-fuzzy convex structures,  the concepts of   quotient structures, substructures and   products  are  presented  and  their fundamental  properties are discussed.

 % in  \cite{{Rosa1},{Rosa2}}, and  an   $L$-convex structure in \cite{Maruyama}.
 %$(X,\mathcal{C})$ where $\mathcal{C}$  and  was defined as a crisp family of   $L^{X}$  satisfying there axioms.,The  One of  morphisms of convex structures,
The notion of convexity preserving functions
 is also generalized to  lattice-valued  fuzzy setting and then an $(L,M)$-fuzzy convexity preserving function is obtained.  Thus there are two categories  $\mathbf{LMCS}$ and $\mathbf{MYCS}$,  where  $\mathbf{LMCS}$  consists of all $(L,M)$-fuzzy convex structures and of all  $(L,M)$-fuzzy convexity preserving functions,  and $\mathbf{MYCS}$  consists of all  $M$-fuzzifying convex structures and of all   $M$-fuzzifying convexity preserving functions.
 %, that is, an $(L,M)$-fuzzy convexity preserving function.
 %Under the framework   of  $(L,M)$-fuzzy convex structures,  the concept of   quotient structures  is presented  and  its fundamental  properties are discussed.
 %$(L,M)$-fuzzy convex structures and their $(L,M)$-fuzzy convexity preserving functions form a category which is denoted by $\mathbf{LMCS}$.  $M$-fuzzifying convex structures and their $M$-fuzzifying convexity preserving functions form a category which is denoted by $\mathbf{MYCS}$.
  Moreover, we  create a functor $\omega$ from $\mathbf{MYCS}$
 to $\mathbf{LMCS}$ and show that there exists an adjunction
between $\mathbf{MYCS}$ and  $\mathbf{LMCS}$.
%  and  then $\omega$ has a right-adjoint.

The above facts will be useful to help further investigations and it is possible that  the fuzzification of convex structure  would be applied to some problems related to the theory of abstract convexity  in the future.

%
%  \section*{Acknowledgements}

\end{document}